\documentclass[12pt,a4paper,reqno]{amsart}
\usepackage[english]{babel}
\usepackage{amssymb,latexsym,amsfonts,amsthm,upref,amsmath}
\usepackage[margin=1in]{geometry}
\usepackage[colorlinks=false]{hyperref}
\usepackage{float}
\usepackage{tikz}
\usepackage{mathdots}
\input xy
\xyoption{all}   
\usepackage{comment} 
\hypersetup{urlcolor=blue, linkcolor=red, citecolor=green, filecolor=magenta}

\newtheoremstyle{prim}{}{}{\normalfont}{}{\bfseries}{}{ }{}
\theoremstyle{prim}
\newtheorem{ex}{Example}
\newtheoremstyle{stil}{}{}{\slshape}{}{\bfseries}{}{ }{}
\theoremstyle{stil}
\newtheorem{thm}{Theorem}[section]
\newtheoremstyle{defi}{}{}{}{}{\bfseries}{}{ }{}
\theoremstyle{defi}
\newtheorem{defn}[thm]{Definition}
\theoremstyle{defi}
\newtheorem{rem}[thm]{Remark}
\theoremstyle{stil}
\newtheorem{pro}[thm]{Proposition}
\theoremstyle{stil}
\newtheorem{lem}[thm]{Lemma}
\theoremstyle{stil}
\newtheorem{cor}[thm]{Corollary}

\newtheorem*{thm*}{Theorem}
\newtheorem*{lem*}{Lemma}

\newcommand{\om}{\mathop{\mathrm{Hom}}}
\newcommand{\ndo}{\mathop{\mathrm{End}}}
\newcommand{\sym}{\mathop{\mathrm{Sym}}}
\newcommand{\gauss}[2]{\genfrac{[}{]}{0pt}{}{#1}{#2}}

\newcommand{\spn}{\mathop{\mathrm{span}}}

\newcommand{\wt}{\mathop{\mathrm{wt}}}

\newcommand{\card}{\mathop{\mathrm{card}}}

\numberwithin{equation}{section}

\begin{document}
\frenchspacing

\title[Title]{A note on the zeroth products of Frenkel-Jing operators}
\author{Slaven Ko\v{z}i\'{c}} 
\address[Slaven Ko\v{z}i\'{c}]{Department of Mathematics, University of Zagreb, Zagreb, Croatia}
\email{kslaven@math.hr}



\thanks{The research was supported by Croatian Science Foundation
under the project   2634 ``Algebraic and combinatorial methods in vertex algebra
theory''.  }
 
\begin{abstract}  
Quantum vertex algebra theory,  developed by H.-S. Li, allows us to
apply  zeroth products of Frenkel-Jing operators, corresponding to Drinfeld realization of
$U_q (\widehat{\mathfrak{sl}}_{n+1})$, on the extension of Koyama
vertex operators. As a result, we obtain an infinite-dimensional space and
describe its structure as a  module for the associative algebra 
$U_q (\mathfrak{sl}_{n+1})_z$, a certain quantum analogue of
$U(\mathfrak{sl}_{n+1})$ which we introduce in this paper.
\end{abstract}

\maketitle  


\section*{Introduction} 

Let $\mathfrak{g}$ be a simple Lie algebra with the root lattice $Q$ and weight
lattice $P$.
Denote by $\widehat{\mathfrak{g}}$ the associated (untwisted) affine Lie algebra on the underlying vector space 
$$\widehat{\mathfrak{g}}=\mathfrak{g}\otimes\mathbb{C}[t,t^{-1}] \oplus
\mathbb{C}c,$$
with the bracket relations  defined in a usual way (for details see
\cite{Kac}).
The induced
$\widehat{\mathfrak{g}}$-module
$$V=V_{\widehat{\mathfrak{g}}}(l,0)=U(\widehat{\mathfrak{g}})\otimes_{U(\widehat{\mathfrak{g}}_{(\leq
0)})} \mathbb{C}_l,\quad l\in\mathbb{N},$$ 
has a vertex operator algebra structure
\begin{align*}
V\quad&\to\quad \om (V,V((z)))\\
v\quad&\mapsto\quad Y(v,z)=\sum_{r\in\mathbb{Z}}v_{r} z^{-r-1}
\end{align*}
(for details and notation see \cite{LL}). Denote by $a_r =a(r)$ the
action of $a\otimes t^r$, $a\in\mathfrak{g}$,
$r\in\mathbb{Z}$, on an arbitrary restricted $\widehat{\mathfrak{g}}$-module $W$ and
set $$a_W (z) =  \sum_{r\in\mathbb{Z}}a_r z^{-r-1} \in \om (W,W((z))).$$
Every restricted $\widehat{\mathfrak{g}}$-module $W$ of level $l$ is a module
for vertex algebra $V$, and vice versa, which gives us   correspondence
\begin{align}
a^{(1)}(r_1)\cdots a^{(m)}(r_m)1 \qquad &\longleftrightarrow \qquad
a^{(1)}_W (z)_{r_1}\cdots a^{(m)}_W (z)_{r_m}1_W, \label{arrows}
\end{align}
$a^{(j)}\in\mathfrak{g}$, $r_j\in\mathbb{Z}$, $j\geq 0$,
between $\widehat{\mathfrak{g}}$-module actions and  products of the local
vertex operators.
The quotient of $V_{\widehat{\mathfrak{g}}}(l,0)$ over its (unique) largest proper
ideal $I_{\widehat{\mathfrak{g}}}(l,0)$, $$L_{\widehat{\mathfrak{g}}}(l,0)=V_{\widehat{\mathfrak{g}}}(l,0)/I_{\widehat{\mathfrak{g}}}(l,0),$$
is a simple vertex operator algebra whose inequivalent irreducible modules
$L_{\widehat{\mathfrak{g}}}(l,L(\lambda))$, where $L(\lambda)$ is the irreducible
 $\mathfrak{g}$-module with the integral dominant highest weight
$\lambda$, are exactly level $l$ irreducible highest weight integrable 
$\widehat{\mathfrak{g}}$-modules.

Consider the  space
$$V_L = M(1)\otimes \mathbb{C}\left\{L\right\},$$
where $M(1)$ is the irreducible level $1$ Heisenberg
algebra module and $\mathbb{C}\left\{L\right\}$ is the group algebra
of the  lattice $L=Q,P$. In the construction of vertex operator algebras and
modules associated to even lattices, the space $V_Q$ is equipped with a vertex
algebra structure, while $V_P$ becomes its module. The simple vertex algebra
$V_Q$ is isomorphic to $L_{\widehat{\mathfrak{g}}}(1,0)$ while $V_P$ is a direct sum
of the irreducible $V_Q$-modules.
This construction provides us with the  realizations of
$$x_\alpha (z)=\sum_{r\in\mathbb{Z}} x_{\alpha}(r)z^{-r-1} =Y(e^{\alpha},z)$$
for Chevalley generators $x_{\alpha}\in\mathfrak{g}$, $\alpha\in Q$, as well as
as with the realizations of the intertwining operators.

The main motivation for this research is the fact that the top of every
irreducible highest weight integrable $\widehat{\mathfrak{g}}$-module
$L_{\widehat{\mathfrak{g}}}(l,L(\lambda))$ is the irreducible
$\mathfrak{g}$-module $L(\lambda)$, where the action of $a\in\mathfrak{g}\subset L_{\widehat{\mathfrak{g}}}(l,0)$ is
given as $a_0$.

We consider quantum affine algebra $U_q
(\widehat{\mathfrak{sl}}_{n+1})$, defined in terms of  Drinfeld generators,  
coefficients of formal Laurent series $x_{j}^{\pm}(z)$, $\psi_{j}(z)$,
$\phi_{j}(z)$, $j=1,2,\ldots,n$ (see \cite{D}). Even though there is no
correspondence as in \eqref{arrows} for $U_q (\widehat{\mathfrak{sl}}_{n+1})$, some of the main ingredients of the
abovementioned constructions do exist. For example, Frenkel-Jing realization
of level $1$ integrable highest weight $U_q (\widehat{\mathfrak{sl}}_{n+1})$-modules (see \cite{FJ})
gives us the quantum analogues of the operators $x_{\alpha}(z)=Y(e^{\alpha},z)$,
while Y. Koyama in \cite{Koyama} found  the realization of intertwining 
operator $\mathcal{Y}_{i}(z)=\mathcal{Y}(e^{\lambda_i},z)$, $i=1,2,\ldots,n$.
All of these operators are defined on the space 
$$V=M(1)\otimes\mathbb{C}\left\{P\right\},$$
 the tensor product of level 1
irreducible $q$-Heisenberg algebra module $M(1)$ and the twisted group algebra 
$\mathbb{C}\left\{P\right\}$ of the (classical) weight lattice $P$.

Since the operators $x_{j}^{\pm}(z)$, $\psi_{j}(z)$,
$\phi_{j}(z)$ are not local, we can not multiply them as in the classical case, using the theory of local vertex operators.
 However, they satisfy the
 notion of quasi compatibility, i.e. for any two such operators
 $a(z),b(z)\in\om(V,V((z)))$ there exists a nonzero polynomial $p(z_1,z_2)$ such
 that 
 $$p(z_1,z_2)a(z_1)b(z_2)\in \om(V,V((z_1,z_2))).$$
 H.-S. Li introduced $r$th products, $r\in\mathbb{Z}$, among quasi compatible operators,
 $$Y_{\mathcal{E}}(a(z),z_0)b(z)=\sum_{r\in\mathbb{Z}}
 (a(z)_{r}b(z))z_{0}^{-r-1}\in (\ndo V)[[z_{0}^{\pm 1}, z^{\pm 1}]],$$ and
 proved that a family of such operators (acting on an
arbitrary level $t\in\mathbb{C}$ restricted module) generates a (weak) quantum
vertex algebra (cf. \cite{Li}, \cite{Li2}, \cite{Li3}). We would like to mention
that there exist several other approaches to development of ``quantum vertex
algebra theories''. For more information the reader may consult
\cite{AB}, \cite{Bor}, \cite{EK}, \cite{FR}.

We consider the set
\begin{equation*}
\left\{x_{j}^{\pm}(zq^t), \psi_{j}(zq^t), \phi_{j}(zq^t) :
j=1,2,\ldots,n,\,
t\in\textstyle\frac{1}{2}\mathbb{Z}\right\}\subset\om(V,V((z))).
\end{equation*}
Our goal is to study zeroth
products of these operators acting   on $\mathcal{Y}_{i}(z)$.
Since zeroth products $a(zq^t)_{0}b(z)$ equal to zero for almost all $t$, we
introduce an action ``$_\bullet$'', which  plays a role of zeroth products
from the (classical) vertex algebra theory in this quantum setting. 
We show that the  ``$_\bullet$'' products similar to the right side
of \eqref{arrows} are equal to zero if and only if the (zeroth) products similar
to the left side of \eqref{arrows}, i.e. the products of Chevalley generators,
are equal to zero.

 For a quasi
compatible pair $(a(z),b(z))$ we define
\begin{equation}\label{theaction}
a(z)_\bullet
b(z)=\sum_{t\in\frac{1}{2}\mathbb{Z}}\left(\frac{1}{zq^{t}}\right)^{(\wt
a(z),\wt b(z))+1}q^t a(zq^t)_{0}b(z).
\end{equation}
 Our main point of interest is a $\left<\mathcal{Y}_i (z)\right>$, a vector
 space over $\mathbb{C}(q^{1/2})$ spanned by all the vectors obtained by the above defined
action ``$_\bullet$" of the operators $x_{j}^{\pm}(z)$, $\psi_{j}(z)$,
$\phi_{j}(z)$, $j=1,2,\ldots,n$, on $\mathcal{Y}_{i}(z)$. Having in mind the
classical case, one may expect $\left<\mathcal{Y}_i (z)\right>$ to be, roughly
speaking, a quantum version of the irreducible $\mathfrak{sl}_{n+1}$-modules
$L(\lambda _i)$.

\makeatletter
\def\thmhead@plain#1#2#3{%
  \thmname{#1}\thmnumber{\@ifnotempty{#1}{ }\@upn{#2}}%
  \thmnote{ {\the\thm@notefont#3}}}
\let\thmhead\thmhead@plain
\makeatother   

In the first section we recall some preliminary results, while
in
the second section we study the action \eqref{theaction} of Frenkel-Jing
operators. Our key result here is the following Lemma, which plays an important
role in the proof of our main result, given in the third section:
\begin{lem*}[\textbf{\ref{drinfeld}}]
For any $a(z)\in\left<\mathcal{Y}_i (z)\right>$ we have
\begin{equation*}
x_{j_1}^{+}(z)_{\bullet}
x_{j_2}^{-}(z)_\bullet
a(z)
-x_{j_2}^{-}(z)_{\bullet}
x_{j_1}^{+}(z)_\bullet
a(z)
=\frac{\delta_{j_1\,j_2}}{q-q^{-1}}\left(\psi_{j_1}(z)-\phi_{j_1}(z)\right)_\bullet
a(z).
\end{equation*}
\end{lem*}
We also identify a basis $\mathcal{B}_i$
for
the space $\left<\mathcal{Y}_i (z)\right>$ as a set, which consists of some of
the operators
$$x_{l_1}^{-}(z)_\bullet  \cdots x_{l_r}^{-}(z)_\bullet \mathcal{Y}_i (z)
\psi_{j_1}(zq^{t_1})\cdots\psi_{j_s}(zq^{t_s}).$$

In the last, third section we introduce
a certain $\mathbb{C}(
q^{1/2})$ subalgebra $U_{q}(\mathfrak{sl}_{n+1})_z$ of the algebra 
$U_{q}(\mathfrak{sl}_{n+1})[[z_1,\ldots,z_n]]$, given in terms of generators
$\bar{e}_j$, $f_j$, $\bar{k}_j$, $j=1,2,\ldots,n$.
Roughly speaking, this new algebra may be considered as a quantum analogue of
$U(\mathfrak{sl}_{n+1})$, since its classical limit is equal to
$U(\mathfrak{sl}_{n+1})$. Next, we construct some infinite-dimensional
$U_{q}(\mathfrak{sl}_{n+1})_z$-modules $L(\lambda_i)_z$ corresponding to
the (finite-dimensional) irreducible
$U_{q}(\mathfrak{sl}_{n+1})$-modules $L(\lambda_i)$ with the integral dominant
highest weight $\lambda_i$. The main result of this paper is Theorem \ref{main},
which identifies $\left<\mathcal{Y}_i (z)\right>$ as a module for
$U_{q}(\mathfrak{sl}_{n+1})_z$ and establishes an
$U_{q}(\mathfrak{sl}_{n+1})_z$-module isomorphism $L(\lambda_i)_z \cong
\left<\mathcal{Y}_i (z)\right>$:

\makeatletter
\def\thmhead@plain#1#2#3{%
  \thmname{#1}\thmnumber{\@ifnotempty{#1}{ }\@upn{#2}}%
  \thmnote{ {\the\thm@notefont#3}}}
\let\thmhead\thmhead@plain
\makeatother

\begin{thm*}[\textbf{\ref{main}}]
(1) There exists a structure of $U_{q}(\mathfrak{sl}_{n+1})_z$-module on the space $\left<\mathcal{Y}_i (z)\right>$ such that
\begin{align*}
&\bar{e}_j a(z) = x_{j}^{+}(z)_\bullet a(z),\\
&f_j a(z)=x_{j}^{-}(z)_\bullet a(z),\\
&\bar{k}_{j}a(z)=\psi_{j}(z)_\bullet a(z)
\end{align*}
for all $j=1,2,\ldots,n$ and $a(z)\in \left<\mathcal{Y}_i (z)\right>$.

\noindent (2) $U_{q}(\mathfrak{sl}_{n+1})_z$-modules $L(\lambda_i)_z$ and $\left<\mathcal{Y}_i (z)\right>$ are isomorphic.
\end{thm*}

\section{Preliminaries}
\subsection{Quantum affine algebra
\texorpdfstring{$U_{q}(\widehat{\mathfrak{sl}}_{n+1})$}{Uq(sl_n+1)}} We recall
some facts from the theory of affine Kac-Moody Lie algebras (see \cite{Kac} for more details).
Let $\widehat{A}=(a_{ij})_{i,j=0}^{n}$ be a generalized Cartan matrix of (affine)
type $A_{n}^{(1)}$ associated with the affine Kac-Moody Lie algebra
$\widehat{\mathfrak{sl}}_{n+1}$.
Let $\widehat{\mathfrak{h}}\subset \widehat{\mathfrak{sl}}_{n+1}$ be a vector space
over $\mathbb{C}$ with a basis
consisting of simple coroots $\alpha^{\vee}_{j}$, $j=0,1,\ldots,n$, and
derivation $d$.
Denote by $\alpha_{0}, \alpha_{1}, \ldots, \alpha_{n}$ simple roots,
i.e. linear functionals from $\widehat{\mathfrak{h}}^{*}$  such that
$$\alpha_{i}(\alpha_{j}^{\vee})=a_{ji},\quad\alpha_{i}(d)=\delta_{i0},\quad
i,j=0,1,\ldots,n.$$ 
The invariant symmetric bilinear form on
$\widehat{\mathfrak{h}}^{*}$ is given by $$(\alpha_{i},\alpha_{j})=a_{ij},\quad
(\delta,\alpha_{i})=(\delta,\delta)=0,\quad i,j=0,1,\ldots,n.$$

Denote by $\Lambda_{0}, \Lambda_{1}, \ldots, \Lambda_{n}$
fundamental weights,  elements of $\widehat{\mathfrak{h}}^{*}$ such that
$$\Lambda_{i}(\alpha_{j}^{\vee})=\delta_{ij},\quad
\Lambda_{i}(d)=0,\quad i,j=0,1,\ldots,n.$$
The center of the Lie algebra $\widehat{\mathfrak{sl}}_{n+1}$ is one-dimensional and it is generated by the element
$$c=\alpha_{0}^{\vee}+\alpha_{1}^{\vee}+\ldots+\alpha_{n}^{\vee}\in \widehat{\mathfrak{h}}$$
and imaginary roots of $\widehat{\mathfrak{sl}}_{n+1}$ are integer multiples of
$$\delta=\alpha_{0}+\alpha_{1}+\ldots+\alpha_{n}\in \widehat{\mathfrak{h}}^{*}.$$
Define a weight lattice $\widehat{P}$ as a free Abelian group generated by 
the elements $\Lambda_{0}, \Lambda_{1}, \ldots, \Lambda_{n}$ and $\delta$.
An integral dominant weight is any   $\Lambda\in \widehat{P}$
such that  $\Lambda(\alpha_{i}^{\vee})\in\mathbb{Z}_{\geq 0}$  for
$i=0,1,\ldots,n$.

Let $\mathfrak{h}\subset\widehat{\mathfrak{h}}$ be Cartan subalgebra of the simple
Lie algebra $\mathfrak{sl}_{n+1}$,  generated by the elements
$\alpha_{1}^{\vee}$, $\alpha_{2}^{\vee}$, \ldots, $\alpha_{n}^{\vee}$.
Denote by 
$$Q=\bigoplus_{i=1}^{n}\mathbb{Z}\alpha_{i}\subset\mathfrak{h}^{*}\quad\textrm{and}\quad P=\bigoplus_{i=1}^{n}\mathbb{Z}\lambda_{i}\subset\mathfrak{h}^{*}$$ 
the classical root lattice and the classical weight lattice, where elements $\lambda_{i}$ 
satisfy $$\lambda_{i}(\alpha_{j}^{\vee})=\delta_{ij},\quad i,j=1,2,\ldots,n.$$

Fix an indeterminate $q$.
For any two integers $m$ and $k$, $k>0$,   define $q$-integers,
 $$[m]=\frac{q^{m}-q^{-m}}{q-q^{-1}}, $$
 and $q$-factorials,
 $$[0] !=1,\quad [k] !=[k][k-1]\cdots[1].$$
 For all nonnegative integers $m$ and $k$, $m\geq k$, define $q$-binomial coefficients,
 $$\gauss{m}{k}=\frac{[m]!}{[k]![m-k]!}.$$

 \begin{defn}\label{quantum}
 The quantized enveloping algebra $U_{q}(\mathfrak{sl}_{n+1})$ is the
 associative algebra over $\mathbb{C}(q^{1/2})$ with unit $1$ generated by the elements
$e_i$, $f_i$ and $K_{i}^{\pm 1}$, $i=1,2,\ldots,n$, subject to the following relations:
\begin{align}
& K_i K_j=K_j K_i,\quad K_i K_{i}^{-1}=K_{i}^{-1}K_{i}=1,\tag{Q1}\label{Q1}\\
&K_i e_j K_{i}^{-1}=q^{(\alpha_i,\alpha_j)}e_j,\quad K_i f_j K_{i}^{-1}=q^{-(\alpha_i,\alpha_j)}f_j,\tag{Q2}\label{Q2}\\
&[e_i,f_j]=\delta_{ij}\frac{K_i - K_{i}^{-1}}{q -q^{-1}},\tag{Q3}\label{Q3}\\
&\sum_{s=0}^{m}(-1)^s \gauss{m}{s} e_{i}^{m-s} e_j e_{i}^{s}=0,\quad\sum_{s=0}^{m}(-1)^s \gauss{m}{s} f_{i}^{m-s} f_j f_{i}^{s}=0 \quad\text{for }i\neq j,\tag{Q4}\label{Q4}
\end{align}
where $m=1-a_{ij}$.
 \end{defn}

Next, we  present  Drinfeld realization of the quantum affine algebra
$U_{q}(\widehat{\mathfrak{sl}}_{n+1})$.

\begin{defn}[(\cite{D})]\label{drinfeld}
The quantum affine algebra $U_{q}(\widehat{\mathfrak{sl}}_{n+1})$ is the associative algebra over $\mathbb{C}(q^{1/2})$ with unit $1$ generated by the 
elements
$x_{i}^{\pm}(r)$, $a_{i}(s)$, $K_{i}^{\pm 1}$, $\gamma^{\pm 1/2}$ and $q^{\pm d}$, $i=1,2,\ldots,n$, $r,s\in\mathbb{Z}$, $s\neq 0$, subject to the following relations:
\begin{align}
& [\gamma^{\pm1/2},u]=0\textrm{ for all }u\in U_{q}(\widehat{\mathfrak{g}}),\tag{D1}\label{D1}\\
& K_i K_j=K_j K_i,\quad K_i K_{i}^{-1}=K_{i}^{-1}K_{i}=1,\tag{D2}\label{D2}\\
& [a_{i}(k),a_{j}(l)]=\delta_{k+l\hspace{2pt}0}\frac{[a_{ij}k]}{k}\frac{\gamma^{k}-\gamma^{-k}}{q-q^{-1}},\tag{D3}\label{D3}\\
& [a_{i}(k),K_{j}^{\pm 1}]=[q^{\pm d},K_{j}^{\pm 1}]=0,\tag{D4}\label{D4}\\
& q^{d}x_{i}^{\pm}(k)q^{-d}=q^{k}x_{i}^{\pm}(k),\quad q^{d}a_{i}(k)q^{-d}=q^{k}a_{i}(k),\tag{D5}\label{D5}\\
& K_{i}x_{j}^{\pm}(k)K_{i}^{-1}=q^{\pm (\alpha_{i},\alpha_{j})}x_{j}^{\pm }(k) ,\tag{D6}\label{D6}\\
& [a_{i}(k),x_{j}^{\pm}(l)]=\pm\frac{[a_{ij}k]}{k}\gamma^{\mp
|k|/2}x_{j}^{\pm}(k+l),\tag{D7}\label{D7}\\
& x_{i}^{\pm}(k+1)x_{j}^{\pm}(l)-q^{\pm(\alpha_i,\alpha_j)}x_{j}^{\pm}(l)x_{i}^{\pm}(k+1)\nonumber\\
&\hspace{20pt}=q^{\pm(\alpha_i,\alpha_j)}x_{i}^{\pm}(k)x_{j}^{\pm}(l+1)-x_{j}^{\pm}(l+1)x_{i}^{\pm}(k),\tag{D8}\label{D8}\\
& [x_{i}^{+}(k),x_{j}^{-}(l)]=\frac{\delta_{ij}}{q-q^{-1}}\left(\gamma^{\frac{k-l}{2}}\psi_{i}(k+l)-\gamma^{\frac{l-k}{2}}\phi_{i}(k+l)\right),\tag{D9}\label{D9}\\
& \sym_{l_1,l_2,\ldots,l_m}\sum_{s=0}^{m}(-1)^{s}\gauss{m}{s}x_{i}^{\pm}(l_{1})\cdots x_{i}^{\pm}(l_{s})x_{j}^{\pm}(k)x_{i}^{\pm}(l_{s+1})\cdots x_{i}^{\pm}(l_{m})=0 \quad\textrm{for }i\neq j,\tag{D10}\label{D10}
\end{align}
where $m=1-a_{ij}$.
The elements $\phi_{i}(-r)$ and $\psi_{i}(r)$, $r\in\mathbb{Z}_{\geq 0}$, are given by \\
\begin{align*}
& \phi_{i}(z)=\sum_{r=0}^{\infty}\phi_{i}(-r)z^{r}=K_{i}^{-1}\exp\left(-(q-q^{-1})\sum_{r=1}^{\infty}a_{i}(-r)z^{r}\right),\\
& \psi_{i}(z)=\sum_{r=0}^{\infty}\psi_{i}(r)z^{-r}=K_{i}\exp\left((q-q^{-1})\sum_{r=1}^{\infty}a_{i}(r)z^{-r}\right).
\end{align*}
\end{defn}

Denote by $x_{i}^{\pm}(z)$ the series
\begin{equation}\label{101_exp:series}
x_{i}^{\pm}(z)=\sum_{r\in\mathbb{Z}}x_{i}^{\pm}(r)z^{-r-1}\in U_{q}(\mathfrak{\widehat{\mathfrak{sl}}_{n+1}})[[z^{\pm 1}]].
\end{equation}
We shall continue to use the notation $x_{i}^{\pm}(z)$ for the action of the expression (\ref{101_exp:series}) on an arbitrary
$U_{q}(\widehat{\mathfrak{sl}}_{n+1})$-module $V$:
$$x_{i}^{\pm}(z)=\sum_{r\in\mathbb{Z}}x_{i}^{\pm}(r)z^{-r-1}\in (\ndo V)[[z^{\pm 1}]].$$
Some basic facts about  $U_{q}(\widehat{\mathfrak{sl}}_{n+1})$ (and its representation
theory) can be found in \cite{HK}.


\subsection{Frenkel-Jing realization}We present  Frenkel-Jing realization of  the integrable highest weight
 $U_{q}(\widehat{\mathfrak{sl}}_{n+1})$-modules $L(\Lambda_{i})$,
 $i=0,1,\ldots,n$ (see \cite{FJ}).

Let $V$ be an arbitrary $U_{q}(\widehat{\mathfrak{sl}}_{n+1})$-module of
level $c$.
The Heisenberg algebra  $U_{q}(\widehat{\mathfrak{h}})$ of level $c$ is generated by the elements $a_{i}(k)$, $i=1,2,\ldots,n$, $k\in\mathbb{Z}\setminus\left\{0\right\}$ and the central element $\gamma^{\pm 1}=q^{\pm c}$ subject to the relations 
\begin{equation}\label{104_heisenberg}
[a_{i}(r),a_{j}(s)]=\delta_{r+s\hspace{2pt}0}\frac{[a_{ij}r][cr]}{r},\quad
i,j=1,2,\ldots,n,\, r,s\in\mathbb{Z}\setminus\left\{0\right\}.
\end{equation}


The Heisenberg algebra $U_{q}(\widehat{\mathfrak{h}})$ has a natural realization on the space $\sym(\widehat{\mathfrak{h}}^{-})$
of the symmetric algebra generated by the elements $a_{i}(-r)$, 
  $r\in\mathbb{Z}_{>0}$, $i=1,2,\ldots,n$, via the following rule:
  \begin{align*}
\gamma^{\pm 1}\hspace{5pt}&\ldots\hspace{5pt}\textrm{multiplication by }q^{\pm c},\\
a_{i}(r)\hspace{5pt}&\ldots\hspace{5pt}\textrm{differentiation operator subject to (\ref{104_heisenberg})},\\
a_{i}(-r)\hspace{5pt}&\ldots\hspace{5pt}\textrm{multiplication by the element }a_{i}(-r).
\end{align*}
Denote the resulted irreducible $U_{q}(\widehat{\mathfrak{h}})$-module
 as $K(c)$. Define the following operators on $K(c)$:
 \begin{align*}
 &E_{-}^{\pm}(a_{i},z)=\exp\left(\mp\sum_{r\geq 1}\frac{q^{\mp cr/2}}{[cr]}a_{i}(-r)z^{r}\right),\\
  &E_{+}^{\pm}(a_{i},z)=\exp\left(\pm\sum_{r\geq 1}\frac{q^{\mp cr/2}}{[cr]}a_{i}(r)z^{-r}\right).
\end{align*}

The associative algebra $\mathbb{C}\left\{P\right\}$ 
is generated by the elements $e^{\alpha_2},e^{\alpha_{3}},\ldots e^{\alpha_{n}}$ and $e^{\lambda_{n}}$ subject to the relations
\begin{align*}
e^{\alpha_{i}}e^{\alpha_{j}}=(-1)^{(\alpha_i,\alpha_j)}e^{\alpha_{j}}e^{\alpha_{i}},\quad
e^{\alpha_{i}}e^{\lambda_{n}}=(-1)^{\delta_{in}}e^{\lambda_{n}}e^{\alpha_{i}},\qquad
i,j=2,3,\ldots,n.
\end{align*}
For 
$\alpha=m_{2}\alpha_2+\ldots+m_{n}\alpha_{n}+m_{n+1}\lambda_{n}\in P$
denote $e^{m_{2}\alpha_2}\cdots e^{m_{n}\alpha_{n}}e^{m_{n+1}\lambda_{n}}\in \mathbb{C}\left\{P\right\}$ by
 $e^{\alpha}$. 
Denote by $\mathbb{C}\left\{Q\right\}$ the subalgebra of $\mathbb{C}\left\{P\right\}$ generated by the elements $e^{\alpha_i}$, $i=1,2,\ldots,n$.
Set
$$\mathbb{C}\left\{Q\right\}e^{\lambda_i}=\left\{ae^{\lambda_{i}}\hspace{2pt}|\hspace{2pt}a\in\mathbb{C}\left\{Q\right\}\right\}.$$
For $\alpha\in Q$ define an action  $z^{\partial_\alpha}$ on $\mathbb{C}\left\{Q\right\}e^{\lambda_i}$ by
$$z^{\partial_\alpha}e^{\beta}e^{\lambda_i}=z^{(\alpha,\beta+\lambda_i)}e^{\beta}e^{\lambda_i}.$$

\begin{thm}[(\cite{FJ})]
 By the action 
\begin{align*}
x_{j}^{\pm}(z)&:=E_{-}^{\pm}(-a_{j},z)E_{+}^{\pm}(-a_{j},z)\otimes e^{\pm\alpha_{j}}z^{\pm\partial_{\alpha_{j}}},
\end{align*}
$j=1,2,\ldots,n$, 
the space $$K(1)\otimes\mathbb{C}\left\{Q\right\}e^{\lambda_i}$$ becomes
the integrable highest weight module of $U_{q}(\widehat{\mathfrak{sl}}_{n+1})$  with the highest weight $\Lambda_i$.
\end{thm}

\subsection{Operator \texorpdfstring{$\mathcal{Y}_i(z)$}{Yi(z)}}
In \cite{Koyama} Koyama found a realization of vertex operators for  level one integrable
highest weight modules of $U_{q}(\widehat{\mathfrak{sl}}_{n+1})$. In
\cite{Kozic} we used  similar operators $\mathcal{Y}_{i}(z)=\mathcal{Y}(e^{\lambda_i},z)$, which we now briefly recall.
Let
$$V=K(1)\otimes \mathbb{C}\left\{P\right\}.$$
For $i=1,2,\ldots,n$ and $l\in\mathbb{Z}$, $l\neq 0$, we define elements
$a_{i}^{*}(l)\in U_{q}(\widehat{\mathfrak{h}})$ by
\begin{align*}
a_{i}^{*}(l)=m_{i}^{(1)}a_{1}(l)+m_{i}^{(2)}a_{2}(l)+\ldots
+m_{i}^{(n)}a_{n}(l),
\end{align*}
where
\begin{align*}
m_{i}^{(j)}=\left\{\begin{array}{l@{\,\ }l}
 \displaystyle\frac{[jl][(n-i+1)l]}{[(n+1)l][l]} &\textrm{ for } j\leq
 i;\vspace{5pt}\\
 \displaystyle\frac{[il][(n-j+1)l]}{[(n+1)l][l]}  &\textrm{ for }j>i.
\end{array}\right.\nonumber
\end{align*}
For $i,j=1,2,\ldots,n,\, l,k\in\mathbb{Z},\, l,k\neq 0$ we have
\begin{equation*}
[a_{i}^{*}(l),a_{j}(k)]=\delta_{ij}\delta_{l+k\hspace{2pt}0}\frac{[l]^{2}}{l}.
\end{equation*}
Define the following operators on the space $V$:
\begin{align*}
E_{-}(a_{i}^{*},z)&=\exp\left(\sum_{r=1}^{\infty}\frac{q^{r/2}}{[r]}a_{i}^{*}(-r)z^{r}\right),\\
E_{+}(a_{i}^{*},z)&=\exp\left(-\sum_{r=1}^{\infty}\frac{q^{r/2}}{[r]}a_{i}^{*}(r)z^{-r}\right).
\end{align*}

\begin{defn}\label{Y}
We define an operator 
$\mathcal{Y}_i(z)\in \om(V,V((z^{1/(n+1)})))$
by
\begin{align*}
\mathcal{Y}_i(z)=E_{-}(a_{i}^{*},z)E_{+}(a_{i}^{*},z)
\otimes e^{\lambda_{i}}(-1)^{(1-\delta_{in})i \partial_{\lambda_{n}}}z^{\partial_{\lambda_{i}}}.
\end{align*}
\end{defn}
By applying the operator $\mathcal{Y}_i(z)$ on an arbitrary vector $v\in V$, we
get a formal power series in fractional powers $z^{\frac{1}{n+1}}$ of the
variable $z$, that has only a finite number of negative powers.

The relations in the next proposition can be proved by a direct calculation. 
Some of them can be found in \cite{DM} and \cite{Kozic}. 

\begin{pro}\label{relations}
For any  $i,j=1,2,\ldots,n$ the following relations hold on $V=K(1)\otimes
\mathbb{C}\left\{P\right\}$:
{\allowdisplaybreaks\begin{align}
& x_{i}^{-}(z_1)x_{i}^{-}(z_2)=(z_1-z_2)(z_1-q^{2}z_2):x_{i}^{-}(z_1)x_{i}^{-}(z_2):\label{r1}\\
& (z_1-qz_2)x_{i}^{-}(z_1)x_{j}^{-}(z_2)=:x_{i}^{-}(z_1)x_{j}^{-}(z_2):\quad\text{for }|i-j|=1\label{r2}\\
& x_{i}^{-}(z_1)x_{j}^{-}(z_2)=:x_{i}^{-}(z_1)x_{j}^{-}(z_2):\quad\text{for }|i-j|>1\label{r3}\\
& (z_1-qz_2)x_{i}^{-}(z_1)\mathcal{Y}_{i}(z_2)=:x_{i}^{-}(z_1)\mathcal{Y}_{i}(z_2):\label{r4}\\
& x_{i}^{-}(z_1)\mathcal{Y}_{j}(z_2)=:x_{i}^{-}(z_1)\mathcal{Y}_{j}(z_2):\quad\text{for }i\neq j\label{r5}\\
& ( z_1 -q^{-3/2}z_2)\psi_{i}(z_1)x_{i}^{-}(z_2)=(q^{-2}z_1-q^{1/2}z_2)x_{i}^{-}(z_1)\psi_{i}(z_2)\label{r6}\\
& (z_1 -q^{3/2}z_2)\psi_{i}(z_1)x_{j}^{-}(z_2)=(qz_1-q^{1/2}z_2)x_{j}^{-}(z_1)\psi_{i}(z_2)\quad\text{for }|i-j|=1\label{r7}\\
& \psi_{i}(z_1)x_{j}^{-}(z_2)=x_{j}^{-}(z_1)\psi_{i}(z_2)\quad\text{for }|i-j|>1\label{r8}\\
& (z_1-q^{3/2}z_2)\psi_{i}(z_1)\mathcal{Y}_{i}(z_2)=(qz_1-q^{1/2}z_2)\mathcal{Y}_{i}(z_2)\psi_{i}(z_1)\label{r9}\\
& \psi_{i}(z_1)\mathcal{Y}_{j}(z_2)=\mathcal{Y}_{j}(z_2)\psi_{i}(z_1)\quad\text{for }i\neq j\label{r10}\\
& (z_1-qz_2)(z_1-q^{-1}z_2)x_{i}^{+}(z_1)x_{i}^{-}(z_2)=:x_{i}^{+}(z_1)x_{i}^{-}(z_2):\label{r11}\\
& x_{i}^{+}(z_1)x_{j}^{-}(z_2)=(z_1-z_2):x_{i}^{+}(z_1)x_{j}^{-}(z_2):\quad\text{for }|i-j|=1\label{r12}\\
& x_{i}^{+}(z_1)x_{j}^{-}(z_2)=:x_{i}^{+}(z_1)x_{j}^{-}(z_2):\quad\text{for }|i-j|>1\label{r13}\\
& x_{i}^{+}(z_1)\mathcal{Y}_{i}(z_2)=(z_1-z_2):x_{i}^{+}(z_1)\mathcal{Y}_{i}(z_2):\label{r14}\\
& x_{i}^{+}(z_1)\mathcal{Y}_{j}(z_2)=:x_{i}^{+}(z_1)\mathcal{Y}_{j}(z_2):\quad\text{for }i\neq j.\label{r15}
\end{align}}
\end{pro}

\section{Space \texorpdfstring{$\left<\mathcal{Y}_i (z)\right>$}{<Yi(z)>}}
\subsection{The action of Frenkel-Jing operators}
Let $V$ be an arbitrary vector space over $\mathbb{C}(q^{1/2})$. We recall two definitions from \cite{Li}.

\begin{defn}\label{quasi:comp}
An ordered pair $(a(z),b(z))$ in $\om(V,V((z)))$ is said to be quasi compatible if there exist a nonzero polynomial 
$p(z_1,z_2)\in\mathbb{C}(q^{1/2})[z_1,z_2]$
such that
$$p(z_1,z_2)a(z_1)b(z_2)\in\om (V,V((z_1,z_2))).$$
\end{defn}

Denote by $\mathbb{C}(q^{1/2})_{*}(z,z_0)$ the extension of the algebra
$\mathbb{C}(q^{1/2})[[z,z_0]]$ by inverses of nonzero polynomials. Let 
$\iota_{z,z_0}$ be a unique algebra embedding $$\iota_{z,z_0}\colon
\mathbb{C}(q^{1/2})_{*}(z,z_0)\to \mathbb{C}(q^{1/2})((z))((z_0))$$
that extends the identity endomorphism of $\mathbb{C}(q^{1/2})[[z,z_0]]$.

\begin{defn}\label{n:product}
Let $(a(z),b(z))$ be a quasi compatible (ordered) pair in $\om(V,V((z)))$. For $r\in\mathbb{Z}$
 we define
$a(z)_{r}b(z)\in (\ndo V)[[z^{\pm 1}]]$ in terms of generating function
$$Y_{\mathcal{E}}(a(z),z_0)b(z)=\sum_{r\in\mathbb{Z}} (a(z)_{r}b(z))z_{0}^{-r-1}\in (\ndo V)[[z_{0}^{\pm 1}, z^{\pm 1}]]$$
by
\begin{align}\label{ramones}
& Y_{\mathcal{E}}(a(z),z_0)b(z)
=\iota_{z,z_0}\left(p(z_0 +
z,z)^{-1}\right)(p(z_1,z)a(z_1)b(z))\bigl.
\bigr|_{z_1 = z+z_{0}},
\end{align}
where $p(z_1,z_2)\in\mathbb{C}(q^{1/2})[z_1,z_2]$ is any nonzero polynomial
such that
\begin{equation*}
p(z_1,z_2)a(z_1)b(z_2)\in \om(V,V((z_1,z_2))).
\end{equation*}
\end{defn}

The Proposition \ref{relations} implies that the each ordered pair consisting of the operators 
$x_{j}^{\pm}(z)$, $\psi_{j}(z)$, $\phi_{j}(z)$  is
quasi compatible, so the $r$th product, $r\in\mathbb{Z}$, (given by Definition
\ref{n:product}) is well defined for such a pair. 

\begin{rem}
The expressions \eqref{r4}, \eqref{r5}, \eqref{r9}, \eqref{r10}, \eqref{r14},
\eqref{r15} in Proposition \ref{relations} are not elements of
$\om(V,V((z_1,z_2)))$, because they consist of rational powers of $z_2$.
However, they are obviously contained in
$\om(V,V((z_1,\textstyle z_{2}^{1/(n+1)})))$,
so Definition \ref{n:product} may be applied. More precisely, 
the $r$th products
$a(z)_{r}\mathcal{Y}_{i}(z)$ for $a(z)=x_{j}^{\pm}(z),\psi_{j}(z),\phi_{j}(z)$
are  well defined elements of the space
$\om(V,V((z^{1/(n+1)})))$.
\end{rem}

By taking into account this remark, from now on, we shall call a pair 
$(a(z),b(z))$, where $a(z) \in \om(V,V((z)))$, $b(z)\in
\om(V,V((z^{1/(n+1)})))$, quasi compatible if there exists a nonzero polynomial $p(z_1,z_2)$ such that 
$$p(z_1,z_2)a(z_1)b(z_2)\in \om (V,V((z_1,z_{2}^{1/(n+1)}))).$$

\begin{rem}For an arbitrary polynomial $p(z_1,z_2)\neq 0$ there exist
only finitely many $t\in\frac{1}{2}\mathbb{Z}$ such that $p(zq^t ,z)=0$.
Hence, for any quasi compatible pair $(a(z),b(z))$ in $\om (V,V((z^{1/(n+1)})))$ there are only
finitely many $t$ such that $a(zq^{t})_{0}b(z)\neq 0$. 
\end{rem}

Set
$$V=K(1)\otimes \mathbb{C}\left\{P\right\}.$$
For $a(z)\in\om (V,V((z^{1/(n+1)})))$ we shall write
$$\wt a(z) = \alpha\in P$$
if $a(z)1=b(z)\otimes e^{\alpha}$ for some $b(z)\in K(1)((z^{1/(n+1)}))$. For example,
\begin{align*}
\wt x_{i}^{\pm}(z) =\pm\alpha_{i},\quad \wt \psi_{i}=\wt\phi_{i}=0,\quad\wt\mathcal{Y}_{i}(z)=\lambda_i.
\end{align*}

\begin{defn}\label{dot_def}
Let $(a(z),b(z))$ be a quasi compatible pair. Define
\begin{equation}\label{dot}
a(z)_\bullet b(z)=\sum_{t\in\frac{1}{2}\mathbb{Z}}\left(\frac{1}{zq^{t}}\right)^{(\wt a(z),\wt b(z))+1}q^t a(zq^t)_{0}b(z).
\end{equation}
\end{defn}
Denote by $\left<\mathcal{Y}_i (z)\right>$ a vector space over
$\mathbb{C}(q^{1/2})$ spanned by all the vectors obtained by the above defined
action of the operators $x_{j}^{\pm}(z)$, $\psi_{j}(z)$, $\phi_{j}(z)$ on $\mathcal{Y}_{i}(z)$:
\begin{align*}
&\left<\mathcal{Y}_i
(z)\right> = \spn_{\mathbb{C}(q^{1/2})}
\left\{
a_{1}(z)_\bullet \ldots_\bullet
a_{k}(z)_\bullet \mathcal{Y}_i (z)\, :\, k\geq 0,\, a_{s}=x_{j}^{\pm}(z),
\psi_{j}(z),\phi_{j}(z),\right. \\
&\hspace{213pt} \left. s=1,2,\ldots,k,\, j=1,2,\ldots,n \right\}.
\end{align*}

\subsection{Space \texorpdfstring{$\left<\mathcal{Y}_i (z)\right>$}{<Yi(z)>}}
First, we shall consider an action of the operators $x_{j}^{-}(z)$ on
$\mathcal{Y}_i (z)$.
Denote by $L(\lambda_i)$ the irreducible highest weight $U_q(\mathfrak{sl}_{n+1})$-module
with the highest weight $\lambda_i$ and with the highest weight vector $v_{\lambda_i}$.

\begin{rem}
On every $L(\lambda_i)$ we have 
$f_{j_k}\ldots f_{j_2}f_{j_1}v_{\lambda_{i}}\neq 0$ if and only if $f_{j_{k-1}}\ldots f_{j_2}f_{j_1}v_{\lambda_{i}}\neq 0$ 
and 
\begin{equation}\label{-1:cond}
(\lambda_i-\alpha_{j_1}-\ldots-\alpha_{j_{k-1}},-\alpha_{j_k})=-1.
\end{equation}
Similarly, 
$e_{j_k}f_{j_{k-1}}\ldots f_{j_2}f_{j_1}v_{\lambda_{i}}\neq 0$ if and only if $f_{j_{k-1}}\ldots f_{j_2}f_{j_1}v_{\lambda_{i}}\neq 0$ 
and 
\begin{equation}\label{-1:cond+}
(\lambda_i-\alpha_{j_1}-\ldots-\alpha_{j_{k-1}},\alpha_{j_k})=-1.
\end{equation}
\end{rem}

\begin{lem}\label{x-Y}
Let $k\in\mathbb{N}$. Then 
$$x_{j_k}^{-}(z)_\bullet \ldots_\bullet x_{j_2}^{-}(z)_\bullet x_{j_1}^{-}(z)_\bullet \mathcal{Y}_i (z)\neq 0\quad\text{if and only if}\quad f_{j_k}\ldots f_{j_2}f_{j_1}v_{\lambda_{i}}\neq 0.$$
\end{lem}

\begin{proof}
We shall prove by induction the following statement:
$$f_{j_k}\ldots f_{j_2}f_{j_1}v_{\lambda_{i}}\neq 0$$ 
if and only if there exists a unique sequence $(t_k,\ldots,t_1)$
in $\frac{1}{2}\mathbb{Z}$ such that 
$$x_{j_l}^{-}(q^{t_l}z)_0 x_{j_{l-1}}^{-}(z)_\bullet \ldots_\bullet x_{j_2}^{-}(z)_\bullet x_{j_1}^{-}(z)_\bullet \mathcal{Y}_i (z)\neq 0\quad\text{for all }l=1,2,\ldots,k.$$

Let $k=1$. Then $f_{j_1}v_{\lambda_i}\neq 0$ if and only if
$j_1=i$. For $j_1=i$ relation (\ref{r4}) implies that $x_{i}^{-}(qz)_0 \mathcal{Y}_i (z)\neq 0$ and $x_{i}^{-}(q^t z)_0 \mathcal{Y}_i (z)= 0$ for
$t\neq 1$, so $x_{i}^{-}(z)_\bullet \mathcal{Y}_i (z)\neq 0$. 
For $j_1\neq i$ relation (\ref{r5}) implies that
$x_{j_1}^{-}(q^t z)_0 \mathcal{Y}_i (z)= 0$ for all $t$. Therefore, the
statement holds for $k=1$ and we have a unique sequence $(t_1)=(1)$.

Let $k=2$. Then $f_{j_2}f_{i}v_{\lambda_i}\neq 0$ if and only if $|j_2-i|=1$. For $j_2=i$ relations (\ref{r1}) and (\ref{r4}) imply that 
$x_{i}^{-}(q^t z)_0 x_{i}^{-}( z)_\bullet  \mathcal{Y}_i (z)= 0$ for all $t$. For $|j_2 - i|>1$ relations (\ref{r3}) and (\ref{r5}) imply that
$x_{j_2}^{-}(q^t z)_0 x_{i}^{-}( z)_\bullet  \mathcal{Y}_i (z)= 0$ for all $t$. Hence $x_{j_2}^{-}(z)_\bullet  x_{i}^{-}( z)_\bullet  \mathcal{Y}_i (z)= 0$
 for $|j_2 -i|\neq 1$.
 For $|j_2 -i|= 1$ relations (\ref{r2}) and (\ref{r5}) imply that $x_{j_2}^{-}(q^2 z)_0 x_{i}^{-}( z)_\bullet  \mathcal{Y}_i (z)\neq 0$ and
 $x_{j_2}^{-}(q^t z)_0 x_{i}^{-}( z)_\bullet  \mathcal{Y}_i (z)= 0$ for $t\neq 2$. Therefore, the statement holds for $k=2$ and
 we have a unique sequence $(t_2,t_1)=(2,1)$.
 
 Suppose that our statement   holds for some $k\geq 2$. 
Assume that $f_{j_{k+1}}\ldots f_{j_1}v_{\lambda_{i}}\neq 0$. Let
$$r=\max\left(\left\{l\leq k : j_l =
j_{k+1}\right\}\cup\left\{0\right\}\right).$$

If $r=0$, then $j_{k+1}\not\in\left\{j_k,\ldots,j_1,i\right\}$. By condition (\ref{-1:cond}) we see that there
exists exactly one $l\in\left\{1,2,\ldots,k\right\}$ such that $|j_{k+1}-j_l|=1$. Relations (\ref{r2}), (\ref{r3}) and (\ref{r5}) imply that
\begin{align*}
&x_{j_{k+1}}^{-}(q^{t_l +1}z)_{0}x_{j_k}^{-}(z)_\bullet \ldots x_{j_2}^{-}(z)_\bullet x_{j_1}^{-}(z)_\bullet \mathcal{Y}_i (z)\neq 0,\\
&x_{j_{k+1}}^{-}(q^{t}z)_{0}x_{j_k}^{-}(z)_\bullet \ldots x_{j_2}^{-}(z)_\bullet x_{j_1}^{-}(z)_\bullet \mathcal{Y}_i (z)= 0 \quad \text{for }t\neq t_l +1.
\end{align*}
Therefore, $x_{j_{k+1}}^{-}(z)_\bullet x_{j_k}^{-}(z)_\bullet \ldots x_{j_2}^{-}(z)_\bullet x_{j_1}^{-}(z)_\bullet \mathcal{Y}_i (z)\neq 0$ and we have a unique sequence
$(t_l +1,t_k,t_{k-1},\ldots,t_1)$.

If $r>0$, then condition (\ref{-1:cond}) implies (together with the induction
assumption) that there exist exactly two integers $l_1$ and $l_2$, $k\geq l_2>l_1>r$, such that $t_{l_1}=t_{l_2}=t_{r}+1$, $|j_{k+1}-j_{l_1}|=|j_{k+1}-j_{l_2}|=1$ and $j_{l_1}\neq j_{l_2}$. By considering relations (\ref{r1})--(\ref{r5})
we see that
\begin{align*}
&x_{j_{k+1}}^{-}(q^{t_{l_1} +1}z)_{0}x_{j_k}^{-}(z)_\bullet \ldots x_{j_2}^{-}(z)_\bullet x_{j_1}^{-}(z)_\bullet \mathcal{Y}_i (z)\neq 0,\\
&x_{j_{k+1}}^{-}(q^{t}z)_{0}x_{j_k}^{-}(z)_\bullet \ldots x_{j_2}^{-}(z)_\bullet x_{j_1}^{-}(z)_\bullet \mathcal{Y}_i (z)= 0 \quad \text{for }t\neq t_{l_1} +1.
\end{align*}
Therefore, $x_{j_{k+1}}^{-}(z)_\bullet x_{j_k}^{-}(z)_\bullet \ldots x_{j_2}^{-}(z)_\bullet x_{j_1}^{-}(z)_\bullet \mathcal{Y}_i (z)\neq 0$ and we have a unique sequence
$(t_{l_1} +1,t_k,t_{k-1},\ldots,t_1)$.

Now assume that there exists a unique sequence $(t_{k+1},\ldots,t_1)$
in $\frac{1}{2}\mathbb{Z}$ such that 
$$x_{j_l}^{-}(q^{t_l}z)_0 x_{j_{l-1}}^{-}(z)_\bullet \ldots x_{j_2}^{-}(z)_\bullet x_{j_1}^{-}(z)_\bullet \mathcal{Y}_i (z)\neq 0\quad\text{for all }l=1,2,\ldots,k+1.$$
By induction assumption we have  $f_{j_k}\ldots f_{j_2}f_{j_1}v_{\lambda_{i}}\neq 0$.

If $r=0$, then by relations (\ref{r1})--(\ref{r5}) we conclude that if $i<j_{k+1}$, then $j_{k+1}+1\not\in\left\{j_k,\ldots,j_1,i\right\}$
and that if $i>j_{k+1}$, then $j_{k+1}-1\not\in\left\{j_k,\ldots,j_1,i\right\}$. 
Hence, exactly one of two indices $j_{k+1}\pm 1$ is an element of the set $\left\{j_k,\ldots,j_1,i\right\}$. 
Furthermore, there exists exactly one index $l=1,2,\ldots,k$ such that $j_{k+1}-1=j_l$ or $j_{k+1}+1=j_l$, which implies that condition (\ref{-1:cond}) holds, i.e.
\begin{equation*}
(\lambda_i-\alpha_{j_1}-\ldots-\alpha_{j_{k}},-\alpha_{j_{k+1}})=-1,
\end{equation*}
so $f_{j_{k+1}}\ldots f_{j_2}f_{j_1}v_{\lambda_{i}}\neq 0$.

If $r>0$, then by induction assumption we have
\begin{equation*}
(\lambda_i-\alpha_{j_1}-\ldots-\alpha_{j_{r-1}},-\alpha_{j_{k+1}})=-1.
\end{equation*}
Since
$$x_{j_{k+1}}^{-}(q^{t}z)_{0}x_{j_r}^{-}(z)_\bullet \ldots
x_{j_2}^{-}(z)_\bullet x_{j_1}^{-}(z)_\bullet \mathcal{Y}_i (z)= 0\quad\text{for all }t,$$ by
considering relations (\ref{r1})--(\ref{r5}) we conclude that there exist exactly two integers $l_1$ and $l_2$, $k\geq l_2>l_1>r$, such that $t_{l_1}=t_{l_2}=t_{r}+1$, $|j_{k+1}-j_{l_1}|=|j_{k+1}-j_{l_2}|=1$ and $j_{l_1}\neq j_{l_2}$.
Hence, we have  $(\lambda_i-\alpha_{j_1}-\ldots-\alpha_{j_{k}},-\alpha_{j_{k+1}})=-1$ so the statement of the Lemma follows.
\end{proof}

Notice that for $x_{j_k}^{-}(z)_\bullet \ldots_\bullet x_{j_1}^{-}(z)_\bullet \mathcal{Y}_i (z)\neq 0$ we have
$$\wt x_{j_k}^{-}(z)_\bullet \ldots_\bullet x_{j_1}^{-}(z)_\bullet \mathcal{Y}_i (z) = \lambda_1 -\alpha_{j_1}-\ldots-\alpha_{j_k}=\wt f_{j_k}\ldots f_{j_2}f_{j_1}v_{\lambda_{i}}.$$
Denote by $\mathcal{B}^{-}_{i}$ the set of all
$$x_{j_k}^{-}(z)_\bullet \ldots_\bullet x_{j_1}^{-}(z)_\bullet \mathcal{Y}_i (z)\neq 0,\quad\text{where }j_s=1,2,\ldots,n ,\text{ }s=1,2,\ldots,k ,\text{ }k\in\mathbb{Z}_{\geq 0},$$
and denote by $\left<\mathcal{Y}_i (z)\right>^{-}$ a subspace of $\left<\mathcal{Y}_i (z)\right>$ spanned by $\mathcal{B}^{-}_{i}$.
By Lemma \ref{x-Y} we have
$$\wt \mathcal{B}^{-}_{i} = \wt L(\lambda_i).$$

\begin{rem}
Considering  \eqref{relations} we see that if
$$f_{j_k}\ldots f_{j_{p+1}}f_{j_p}\ldots f_{j_1}v_{\lambda_{i}} =
f_{\sigma(j_k)}\ldots
f_{j_{\sigma(j_{p+1})}}f_{j_p}\ldots f_{j_1}v_{\lambda_{i}}$$ for some
permutation $\sigma$, then
\begin{align*}
&x_{j_k}^{-}(z)_\bullet \ldots_\bullet 
x_{j_{p+1}}^{-}(z)_\bullet x_{j_p}^{-}(z)_\bullet \ldots_\bullet
x_{j_1}^{-}(z)_\bullet \mathcal{Y}_i (z)\\
&\qquad =
x_{\sigma(j_k)}^{-}(z)_\bullet \ldots_\bullet x_{\sigma(j_{p+1})}^{-}(z)_\bullet
x_{j_p}^{-}(z)_\bullet \ldots_\bullet x_{j_1}^{-}(z)_\bullet \mathcal{Y}_i (z).
\end{align*}
 Therefore, we have $$\card \mathcal{B}^{-}_{i} = \dim L(\lambda_i).$$
\end{rem}

\begin{lem}\label{linear_ind}
The set $\mathcal{B}^{-}_{i}$ forms a basis of $\left<\mathcal{Y}_i (z)\right>^{-}$.
\end{lem}

\begin{proof}
Suppose that $$\sum_{s=1}^{k}\nu_s a_s(z)=0$$ for some nonzero scalars $\nu_s$ and $a_s(z)\in \mathcal{B}^{-}_{i}$.
Furthermore, assume that $k\geq 2$ is the smallest positive integer for which
such a nontrivial linear combination exists.
We can choose $p$ such that $\wt a_p (z)$ is maximal, i.e. $\wt a_p (z)\not <
\wt a_r (z)$ for all $r\in\left\{1,2,\ldots,k\right\}$.
Next, we can choose $j_l,\ldots,j_1$ such that
$\wt x_{j_l}^{-}(z)_\bullet \ldots_\bullet x_{j_1}^{-}(z)_\bullet a_p (z)$
is the lowest weight of $L(\lambda_i)$ and, therefore,
$$x_{j_l}^{-}(z)_\bullet \ldots_\bullet x_{j_1}^{-}(z)_\bullet a_p (z)\neq  0.$$ Then the linear
combination $$x_{j_l}^{-}(z)_\bullet \ldots_\bullet x_{j_1}^{-}(z)_\bullet \sum_{s=1}^{k}\nu_s a_s(z)=\sum_{s=1}^{k}\nu_s x_{j_l}^{-}(z)_\bullet \ldots_\bullet x_{j_1}^{-}(z)_\bullet a_s(z)=0$$
consists of less then $k$ nonzero summands. Contradiction.
\end{proof}

Next, we have an analogue of Lemma \ref{x-Y} for the operators $x_{j}^{+}(z)$:
\begin{lem}\label{x+Y}
Let $k\in\mathbb{N}$. Then 
$$x_{j_{k+1}}^{+}(z)_\bullet x_{j_k}^{-}(z)_\bullet \ldots_\bullet x_{j_2}^{-}(z)_\bullet x_{j_1}^{-}(z)_\bullet \mathcal{Y}_i (z)\neq 0\quad\text{if and only if}\quad e_{j_{k+1}}f_{j_k}\ldots f_{j_2}f_{j_1}v_{\lambda_{i}}\neq 0.$$
\end{lem}

\begin{proof}
Let 
\begin{equation*}
r=\max\left(\left\{l : j_l =j_{k+1}\right\}\cup\left\{0\right\}\right).
\end{equation*}
If $r=0$, then 
$e_{j_{k+1}}f_{j_k}\ldots f_{j_1}v_{\lambda_{i}}=0$
 and, by relations (\ref{r12})--(\ref{r15}),
\begin{equation*}
x_{j_{k+1}}^{+}(z)_\bullet x_{j_k}^{-}(z)_\bullet \ldots_\bullet x_{j_2}^{-}(z)_\bullet x_{j_1}^{-}(z)_\bullet \mathcal{Y}_i (z)=0.
\end{equation*}

If $r\neq 0$, then there exists a  unique integer $t_r$ such that 
\begin{align*}
&x_{j_{r}}^{-}(q^{t_r}z)_0 x_{j_{r-1}}^{-}(z)_\bullet \ldots_\bullet x_{j_2}^{-}(z)_\bullet x_{j_1}^{-}(z)_\bullet \mathcal{Y}_i (z)\neq 0,\\
&x_{j_{r}}^{-}(q^{t}z)_0 x_{j_{r-1}}^{-}(z)_\bullet \ldots_\bullet x_{j_2}^{-}(z)_\bullet x_{j_1}^{-}(z)_\bullet \mathcal{Y}_i (z)= 0,\quad\text{for }t\neq t_r
\end{align*}
(cf. the proof of Lemma \ref{x-Y}). If there exists an integer $s$ such that $r<s\leq k$ and $|j_s - j_{r}|=1$, then by (\ref{-1:cond}) and (\ref{-1:cond+}) we get
$e_{j_{k+1}}f_{j_k}\ldots f_{j_1}v_{\lambda_{i}}=0$ and by relations (\ref{r11})--(\ref{r15})
\begin{equation*}
x_{j_{k+1}}^{+}(z)_\bullet x_{j_k}^{-}(z)_\bullet \ldots_\bullet x_{j_2}^{-}(z)_\bullet x_{j_1}^{-}(z)_\bullet \mathcal{Y}_i (z)=0.
\end{equation*}
If $|j_s - j_{r}|>1$ for all integers $s$ such that $r<s\leq k$, then by (\ref{-1:cond}) and (\ref{-1:cond+}) we get
$e_{j_{k+1}}f_{j_k}\ldots f_{j_1}v_{\lambda_{i}}\neq 0$ and relations (\ref{r11})--(\ref{r15}) imply that
\begin{align*}
&x_{j_{k+1}}^{+}(q^{t_r +1}z)_0 x_{j_{k}}^{-}(z)_\bullet \ldots_\bullet x_{j_2}^{-}(z)_\bullet x_{j_1}^{-}(z)_\bullet \mathcal{Y}_i (z)\neq 0,\\
&x_{j_{k+1}}^{+}(q^{t}z)_0 x_{j_{k}}^{-}(z)_\bullet \ldots_\bullet x_{j_2}^{-}(z)_\bullet x_{j_1}^{-}(z)_\bullet \mathcal{Y}_i (z)= 0,\quad\text{for }t\neq t_r +1,
\end{align*}
so the statement of the lemma follows.
\end{proof}

By considering relations (\ref{r6})--(\ref{r10}) and applying the same technique
as in the proof of Lemmas \ref{x-Y} and \ref{x+Y}  one can prove

\begin{lem}\label{nedokazana}
Let $k\in\mathbb{N}$. Then 
\begin{gather*}
\psi_{j_{k+1}}(z)_\bullet x_{j_k}^{-}(z)_\bullet \ldots_\bullet x_{j_2}^{-}(z)_\bullet x_{j_1}^{-}(z)_\bullet \mathcal{Y}_i (z)\neq 0\\
\text{if and only if}\\
e_{j_{k+1}}f_{j_k}\ldots f_{j_2}f_{j_1}v_{\lambda_{i}}\neq 0\quad\text{or}\quad f_{j_{k+1}}f_{j_k}\ldots f_{j_2}f_{j_1}v_{\lambda_{i}}\neq 0.
\end{gather*}
\end{lem}

As before, if $\psi_{j_{k+1}}(z)_\bullet x_{j_k}^{-}(z)_\bullet \ldots_\bullet x_{j_1}^{-}(z)_\bullet \mathcal{Y}_i (z)\neq 0$,
then  there exists a unique $t\in\frac{1}{2}\mathbb{Z}$ such that
$\psi_{j_{k+1}}(zq^t)_0 x_{j_k}^{-}(z)_\bullet \ldots_\bullet x_{j_1}^{-}(z)_\bullet \mathcal{Y}_i (z)\neq 0$. 

Since $\phi_{j}(z)$ contains only nonnegative powers of the variable $z$, we
have

\begin{lem}
For any $a(z)\in \left<\mathcal{Y}_i (z)\right>$ and $j=1,2,\ldots,n$  $$\phi_{j}(z)_{\bullet } a(z)=0.$$ 
\end{lem}

In order to clarify the application of Definition \ref{dot_def} on the
Frenel-Jing operators, we list in the following corollary  
necessary conditions for the summands on the right side of (\ref{dot}) to be
nonzero. Its statement is a consequence of the proofs of Lemmas \ref{x-Y},
\ref{x+Y} and \ref{nedokazana}.

\begin{cor}
Let $j=1,2,\ldots,n$ and
$$a(zq^t)_0 x_{j_k}^{-}(z q^{t_k})_0  \ldots x_{j_1}^{-}(zq^{t_1})_0 
\mathcal{Y}_{i}(z)\neq 0.$$ Set $t_0 =0$ and 
$$r_0 =\max\left(\left\{l : |j_l - j|=1\right\}\cup\left\{0\right\}\right).$$
Then
\begin{align*}
t=\left\{\begin{array}{l@{\,\ }l}
 t_{r_0}+1 &\textrm{ for }a(z)=x_{j}^{-}(z);\\
  t_{r_0}+\frac{3}{2} &\textrm{ for }a(z)=\psi_{j}(z);\\
  t_{r_0}+2 &\textrm{ for }a(z)=x_{j}^{+}(z).
\end{array}\right.\nonumber
\end{align*}
\end{cor}

The next lemma gives us an analogue of Drinfeld relation (\ref{D9}) on the space $\left<\mathcal{Y}_i (z)\right>$.

\begin{lem}\label{drinfeld}
For any $a(z)\in\left<\mathcal{Y}_i (z)\right>$ we have
\begin{equation}\label{def:rel}
x_{j_1}^{+}(z)_{\bullet}
x_{j_2}^{-}(z)_\bullet
a(z)
-x_{j_2}^{-}(z)_{\bullet}
x_{j_1}^{+}(z)_\bullet
a(z)
=\frac{\delta_{j_1\,j_2}}{q-q^{-1}}\left(\psi_{j_1}(z)-\phi_{j_1}(z)\right)_\bullet
a(z).
\end{equation}
\end{lem}

\begin{proof}
By using Lemmas \ref{x-Y} and \ref{x+Y} together with (\ref{-1:cond}) and
(\ref{-1:cond+}), we easily see that $[x_{j_1}^{+}(z),x_{j_2}^{-}(z)]=0$ on 
$\left<\mathcal{Y}_i (z)\right>$ for $j_1\neq j_2$.

Assume $j_1 = j_2$. Set $j=j_1=j_2$ and let $$a(z)=x_{p_{k}}^{-}(z)_\bullet \ldots_\bullet x_{p_1}^{-}(z)_\bullet \mathcal{Y}_i (z)$$
be an element of $\left<\mathcal{Y}_i (z)\right>$ such that $x_{j}^{-}(zq^t)_0 a(z)\neq 0$ for some (unique) $t$. 
By using relations given by Proposition \ref{relations} one can construct polynomials $r_s (z_1,z_2)$, $s=1,2,3$, such that
\begin{align}
& q^t (z_1 - z_2) r_{1}(z_1,z_2)x_{j}^{-}(z_1 q^t) a(z_2) = r_{1}(z_1,z_2) :x_{j}^{-}(z_1 q^t) a(z_2):\label{c:1}\\
& q^{t+1} (z_1 - z_2) r_{2}(z_1,z_2)x_{j}^{+}(z_1 q^{t+1}) x_{j}^{-}(z_2)_\bullet  a(z_2) = r_{2}(z_1,z_2) :x_{j}^{+}(z_1 q^{t+1}) x_{j}^{-}(z_2)_\bullet  a(z_2):\label{c:2}\\
& (z_1 - z_2) r_{3}(z_1,z_2)\psi_{j}(z_1 q^{t+1/2}) a(z_2) = (qz_1 -q^{-1}z_2)r_{3}(z_1,z_2) a(z_2)\psi_{j}(z_1 q^{t+1/2}). \label{c:3}
\end{align} 
Next, we have
\begin{equation}\label{c:4}
:x_{j}^{+}(q^{t+1} z)x_{j}^{-}(q^t z):=\psi_{j}(zq^{t+1/2}),
\end{equation}
so relations (\ref{c:1})--(\ref{c:3}), together with (\ref{c:4}), imply that
equation (\ref{def:rel}) holds.

For $a(z)=x_{p_{k}}^{-}(z)_\bullet \ldots_\bullet x_{p_1}^{-}(z)_\bullet \mathcal{Y}_i (z)$  such  that $x_{j}^{+}(z)_\bullet  a(z)\neq 0$ 
we can proceed similarly and for
$a(z)=x_{p_{k}}^{-}(z)_\bullet \ldots_\bullet x_{p_1}^{-}(z)_\bullet \mathcal{Y}_i (z)$  such  that $x_{j}^{\pm}(z)_\bullet  a(z)= 0$ the statement follows from Lemma \ref{nedokazana}.
\end{proof}

Since $\wt \psi_j (z)=0$ we have $$\wt\left<\mathcal{Y}_i (z)\right>  =\wt \mathcal{B}_{i}^{-}=\wt L(\lambda_i). $$
Every homogeneous vector $a(z)\in \left<\mathcal{Y}_i (z)\right>$ is a linear
combination of the vectors having the form
\begin{equation*}
b(z)\psi_{j_1}(zq^{t_{j_1}})\ldots \psi_{j_k}(zq^{t_{j_k}})
\end{equation*}
for $$b(z)\in \mathcal{B}_{i}^{-}, \,\,\,k\geq 0,\,\,\, 1\leq j_1\leq \ldots\leq j_k\leq n,\,\,\, t_{j_s}\in\textstyle\frac{1}{2}\mathbb{Z},\,\,\, \wt a(z)=\wt b(z).$$
Denote by $\mathcal{B}_{i,\psi}^{-}$ the set of all vectors 
\begin{equation}\label{the_vectors}
b(z)=b^{-}(z)\psi_{j_1}(zq^{t_{1}})\ldots \psi_{j_k}(zq^{t_{k}}),
\end{equation}
where 
\begin{align*}
&b^{-}(z)\in \mathcal{B}_{i}^{-}, \,\,\,k\geq 0,\,\,\, 1\leq j_1\leq
\ldots\leq j_k\leq n,\,\,\, t_{s}\in\textstyle\frac{1}{2}\mathbb{Z},\\
& t_{s}\leq t_{r} \text{ if } s\leq r\text{ and } j_s =j_r.
\end{align*}
Set
\begin{equation}\label{ordering}
o(b(z))=(t_1,\ldots,t_k),\qquad t(b(z))=(j_1,\ldots,j_k).
\end{equation}

Naturally, the space $\left<\mathcal{Y}_i (z)\right>_{\psi}$  spanned by the set $\mathcal{B}_{i,\psi}^{-}$ (over
$\mathbb{C}(q^{1/2})$) is bigger
than $\left<\mathcal{Y}_i (z)\right>$, i.e.
$\left<\mathcal{Y}_i (z)\right>\subsetneq \left<\mathcal{Y}_i (z)\right>_\psi$ . 
For example,
\begin{equation*}
\mathcal{Y}_{1}(z)\psi_{2}(zq^t)\not\in \left<\mathcal{Y}_1
(z)\right>\quad\text{for all } t\in\textstyle\frac{1}{2}\mathbb{Z}
\end{equation*}
because
\begin{equation}\label{notin}
\psi_{2}(z)_{\bullet} \mathcal{Y}_{1}(z) =0.
\end{equation}

\begin{lem}\label{independence_final}
The set $\mathcal{B}_{i,\psi}^{-}$ is linearly independent.
\end{lem}

\begin{proof}
We shall prove a ``stronger'' statement from which the statement of the Lemma clearly follows:
The set of all vectors 
$$b(z)\psi_{j_1}(zq^{t_{1}})\ldots \psi_{j_k}(zq^{t_{k}}),$$
where
\begin{align*}
&b(z)\in \bigcup_{i=1}^{n} \mathcal{B}_{i}^{-}, \,\,\,k\geq 0,\,\,\, 1\leq
j_1\leq \ldots\leq j_k\leq n,\,\,\, t_{s}\in\textstyle\frac{1}{2}\mathbb{Z},\\
& t_{s}\leq t_{r} \quad\text{if}\quad s\leq r\text{ and } j_s =j_r.
\end{align*}is linearly independent.

Suppose that 
\begin{equation}\label{eq:lin_comb}
\sum_{s=1}^{k}\nu_s a_s(z)=0\quad\text{for some
}\nu_s\in\mathbb{C}(q^{1/2})\setminus \left\{0\right\},\, a_s(z)\in
\bigcup_{i=1}^{n}\mathcal{B}_{i,\psi}^{-}.
\end{equation}
Furthermore, assume that $k\geq 2$ is the smallest positive integer for which
such a nontrivial linear combination exists. 
Without loss of generality we can assume that the weights of all $a_s(z)$ are equal. 
Indeed, if $\wt a_r(z)\neq \wt a_s (z)$ for some $r,s=1,2,\ldots,k$, we can
proceed similarly as in the proof of Lemma  \ref{linear_ind}.
Next, we can assume that each $a_s(z)$ is of the form
$$a_s(z)=\mathcal{Y}_i (z) \psi_{j_{1,s}}(zq^{t_{1,s}})\cdots \psi_{j_{l_s,s}}(zq^{t_{l_s,s}})$$
for some $l_s \geq 0$, $1\leq j_{1,s}\leq \ldots\leq j_{l_s,s}\leq n$, $t_{j,s}\in\frac{1}{2}\mathbb{Z}$, $j=1,2,\ldots,l_s$.
Multiplying linear combination (\ref{eq:lin_comb}) by appropriate
invertible operators we can ``remove" $\mathcal{Y}_i (z)$ and get
\begin{equation}\label{eq2:lin_comb}
L(z):=\sum_{s=1}^{k}\nu_s \psi_{j_{1,s}}(zq^{t_{1,s}})\cdots \psi_{j_{l_s,s}}(zq^{t_{l_s,s}})=0.
\end{equation}
Since 
$$e^{\lambda_{m}}\psi_{j}(z)=q^{-\delta_{m,j}}\psi_{j}(z)e^{\lambda_{m}},$$
by multiplying (\ref{eq2:lin_comb}) with $e^{\lambda_m}$, moving the (invertible) operator $e^{\lambda_m}$ all the way to the right
and then dropping the operator we get
\begin{equation}\label{eq3:lin_comb}
\sum_{s=1}^{k}\nu_s q^{-\delta_{m,j_{1,s}}-\ldots-\delta_{m,j_{l_s,s}}} \psi_{j_{1,s}}(zq^{t_{1,s}})\cdots \psi_{j_{l_s,s}}(zq^{t_{l_s,s}})=0.
\end{equation}
From (\ref{eq3:lin_comb}) we see that it is sufficient to consider only a linear combination as in (\ref{eq2:lin_comb}) such that all the powers of $q$,
$-\delta_{m,j_{1,s}}-\ldots-\delta_{m,j_{l_s,s}}$, $s=1,2,\ldots,k$, are equal.
Since the operators $\psi_{j}(z)$ are invertible, we can assume that for each $j$ and $t$ there exists an index $s$ 
such that the operator $\psi_j (zq^t)$ does not appear in the $s$th summand in (\ref{eq2:lin_comb}).
Recall relations (\ref{r9}) and (\ref{r10}). We can choose some $j=1,2,\ldots,n$ and $t\in\frac{1}{2}\mathbb{Z}$ 
such that there exist indices $r$ and $s$ such that
$$\psi_{j_{1,r}}(zq^{t_{1,r}})\cdots \psi_{j_{l_r,r}}(zq^{t_{l_r,r}})_\bullet \mathcal{Y}_{j}(zq^t)=0\quad\text{and}\quad \psi_{j_{1,s}}(zq^{t_{1,s}})\cdots \psi_{j_{l_s,s}}(zq^{t_{l_s,s}})_\bullet \mathcal{Y}_{j}(zq^t)\neq 0.$$
Finally, by using the substitution $z_0 = q^{-t}z$ in $L(z_0)_\bullet  \mathcal{Y}_{j}(z_0 q^t)=0$ we get a contradiction to the choice of $k$.
\end{proof}

By the discussion preceding Lemma \ref{independence_final} we see that the set
$$\mathcal{B}_{i} = \mathcal{B}_{i,\psi}^{-} \cap \left<\mathcal{Y}_i (z)\right>$$
spans $\left<\mathcal{Y}_i (z)\right>$, so as a consequence of Lemma
\ref{independence_final} we have

\begin{thm}\label{easycor}
The set $\mathcal{B}_{i}$ forms a basis for the space $\left<\mathcal{Y}_i
(z)\right>$.
\end{thm}

\begin{ex}\label{exampleeasy}
Consider the quantum affine algebra $U_{q}(\widehat{\mathfrak{sl}}_{2})$. The
basis for the space $\left<\mathcal{Y}_1
(z)\right>$, given in Theorem \ref{easycor}, is the set
$$\mathcal{B}_{1}=A_{1,\lambda_{1}}\cup A_{1,\lambda_{1}-\alpha_1},$$
where
\begin{align*}
&A_{1,\lambda_{1}}=\left\{ \mathcal{Y}_1
(z)\psi_{1}(zq^{3/2})^{l} : l\in\mathbb{Z}_{\geq 0}\right\};
\\
&A_{1,\lambda_{1}-\alpha_1}=\left\{x_{1}^{-}(z)_\bullet \mathcal{Y}_1
(z)\psi_{1}(zq^{3/2})^{l} : l\in\mathbb{Z}_{\geq 0}\right\}.
\end{align*}
\end{ex}

\begin{ex}\label{examplehard}
Consider the quantum affine algebra $U_{q}(\widehat{\mathfrak{sl}}_{3})$. The
basis for the space $\left<\mathcal{Y}_1
(z)\right>$, given in Theorem \ref{easycor}, is the set
$$\mathcal{B}_{1}=A_{2,\lambda_{1}}\cup A_{2,\lambda_{1}-\alpha_1}\cup
A_{2,\lambda_{1}-\alpha_1-\alpha_2},$$ where
\begin{align*}
&A_{2,\lambda_{1}}=\left\{ \mathcal{Y}_1
(z)\psi_{1}(zq^{3/2})^{l}\psi_{2}(zq^{5/2})^{m} : l,m\in\mathbb{Z}_{\geq
0},\,\text{if }m>0\text{ then }l>0\right\};
\\
&A_{2,\lambda_{1}-\alpha_1}=\left\{x_{1}^{-}(z)_\bullet \mathcal{Y}_1
(z)\psi_{1}(zq^{3/2})^{l}\psi_{2}(zq^{5/2})^{m} : l,m\in\mathbb{Z}_{\geq
0}\right\};
\\
&A_{2,\lambda_{1}-\alpha_1-\alpha_2}=\left\{x_{2}^{-}(z)_\bullet
x_{1}^{-}(z)_\bullet \mathcal{Y}_1 (z)\psi_{1}(zq^{3/2})^{l}\psi_{2}(zq^{5/2})^{m} : l,m\in\mathbb{Z}_{\geq
0}\right\}.
\end{align*}
The constraint ``if $m>0$ then
$l>0$'', in the definition of set $A_{2,\lambda_{1}}$,  is a consequence of
\eqref{notin}.
\end{ex}

\section{Algebra \texorpdfstring{$U_{q}(\mathfrak{sl}_{n+1})_z$}{Uq(sln+1)[z]}
and its representations} \subsection{Algebra
\texorpdfstring{$U_{q}(\mathfrak{sl}_{n+1})_z$}{Uq(sln+1)[z]}}
For $j=1,2,\ldots,n$ define  elements  $L_j, M_j\in
U_q(\mathfrak{h})\subset U_q(\mathfrak{sl}_{n+1})$   by
\begin{align*}
&L_j =K_{1}^{n+1-j}K_{2}^{2(n+1-j)}\cdots
K_{j-1}^{(j-1)(n+1-j)}K_{j}^{j(n+1-j)}K_{j+1}^{j(n-j)}\cdots K_{n}^{j},\\
&M_{j}=\left\{\begin{array}{l@{\,\ }l} q^{-(n+1)^2} &\textrm{ if }j=1\text{ or
}j=n,\\
 q^{-(n+1)^2}L_{j-1}L_{j+1}&\text{ if }1<j<n.
\end{array}\right.
\end{align*}
The elements $L_j$ and $M_j$ are invertible and they satisfy
$$L_j L_k =L_k L_j,\qquad M_j M_k =M_k M_j$$
for $j,k=1,2,\ldots,n$.
Let $z_1,\ldots,z_n$ be commutative formal variables and
$$w_j=e^{(M_j-M_{j}^{-1}) z_j}=\sum_{r\geq
0}\frac{(M_{j} -M_{j}^{-1})^r}{r!}z_{j}^{r}\in
U_{q}(\mathfrak{sl}_{n+1})[[z_1,\ldots,z_n]],\quad j=1,2,\ldots,n .$$ The
algebra $U_{q}(\mathfrak{h})[w_1,\ldots,w_n]$ is  commutative but the elements $w_j$ are not central in
the algebra $U_{q}(\mathfrak{sl}_{n+1})[w_1,\ldots,w_n]$.  


Define elements $\bar{e}_j,\bar{k}_j \in U_{q}(\mathfrak{sl}_{n+1})[w_1,\ldots,w_n]$
by $$\bar{e}_j=e_j w_j,\quad \bar{k}_j
 =(K_{j}-K_{j}^{-1})w_j\qquad\text{for }j=1,2,\ldots,n.$$ 
 Denote by
 $U_{q}(\mathfrak{sl}_{n+1})_z$ a $\mathbb{C}(q^{1/2})$ subalgebra of
 $U_{q}(\mathfrak{sl}_{n+1})[w_1,\ldots,w_n]$ generated by the elements
 $$\bar{e}_j,\, f_j,\,  \bar{k}_{j},\quad j=1,2,\ldots,n.$$

\begin{rem}
Since the classical limit $q\to 1$ of  $K_j$ is equal to $1$ (cf. \cite{Lu}),
the classical limit of $w_j$ is also $1$, so the classical limit of
$U_{q}(\mathfrak{sl}_{n+1})_z$ is the universal enveloping algebra
$U(\mathfrak{sl}_{n+1})$. Moreover, classical limits of $\bar{e}_j$, $f_j$ and
$\bar{h}_{j}= \bar{k}_j / (q-q^{-1})$ are exactly Chevalley generators of
$U(\mathfrak{sl}_{n+1})$.
\end{rem}

\begin{rem}
In general, the algebra $U_{q}(\mathfrak{sl}_{n+1})_z$ can not be defined in
terms of closed-form expressions among generators $\bar{e}_j,f_j,k_j$,
$j=1,2,\ldots,n$.
For example,  we have
$$f_{j+1} M_j =q^{n+1} M_{j} f_{j+1}$$
for $2\leq j\leq n-1$ and therefore 
\begin{align*}
f_{j+1} \bar{e}_{j}&=f_{j+1}e_j w_j=e_j f_{j+1} w_j = e_j f_{j+1} \sum_{r\geq
0}\frac{(M_{j} -M_{j}^{-1})^r}{r!}z_{j}^{r}\\
&=e_{j} \left(\sum_{r\geq
0}\frac{(q^{n+1}M_{j} -q^{-n-1}M_{j}^{-1})^r}{r!}z_{j}^{r}\right) f_{j+1}.
\end{align*}
\end{rem}

\subsection{Module \texorpdfstring{$L(\lambda_i)_z$}{L(lambda i)}}

For any $U_{q}(\mathfrak{sl}_{n+1})$-module $V$ the  space
$V[[z_1,\ldots,z_n]]$ can be in a natural way equipped
by the structure of  $U_{q}(\mathfrak{sl}_{n+1})[[z_1,\ldots,z_n]]$-module, as
well as by the structure of $U_{q}(\mathfrak{sl}_{n+1})[w_1,\ldots,w_n]$-module.
 Naturally, every $U_{q}(\mathfrak{sl}_{n+1})[w_1,\ldots,w_n]$-module
is also a $U_{q}(\mathfrak{sl}_{n+1})_z$-module. 

Let $v_{\lambda_i}$ be the highest weight vector of $U_{q}(\mathfrak{sl}_{n+1})$-module $L(\lambda_i)$ with the 
dominant integral highest weight $\lambda_i$.
Denote by $L(\lambda_i)_z$ an $U_{q}(\mathfrak{sl}_{n+1})_z$-submodule of
$L(\lambda_i)[[z_1,\ldots,z_n]]$ generated by
$v_{\lambda_i}$, i.e. $$L(\lambda_i)_z= U_{q}(\mathfrak{sl}_{n+1})_z
v_{\lambda_i}.$$

\begin{thm}\label{main}
(1) There exists a structure of $U_{q}(\mathfrak{sl}_{n+1})_z$-module on the space $\left<\mathcal{Y}_i (z)\right>$ such that
\begin{align}
&\bar{e}_j a(z) = x_{j}^{+}(z)_\bullet a(z),\label{main_1}\\
&f_j a(z)=x_{j}^{-}(z)_\bullet a(z),\label{main_2}\\
&\bar{k}_{j}a(z)=\psi_{j}(z)_\bullet a(z)\label{main_3}
\end{align}
for all $j=1,2,\ldots,n$ and $a(z)\in \left<\mathcal{Y}_i (z)\right>$.

\noindent (2) $U_{q}(\mathfrak{sl}_{n+1})_z$-modules $L(\lambda_i)_z$ and $\left<\mathcal{Y}_i (z)\right>$ are isomorphic.
\end{thm}

\begin{proof}
First, notice that
$$\psi_j (z)_\bullet x_{j_k}^{-}(z)_\bullet \ldots_\bullet
x_{j_2}^{-}(z)_\bullet x_{j_1}^{-}(z)_\bullet \mathcal{Y}_i (z)\neq
0\quad\text{if and only if}\quad \bar{k}_{j}f_{j_k}\ldots
f_{j_2}f_{j_1}v_{\lambda_{i}}\neq 0.$$
For any homogeneous vector $v\in L(\lambda_i)$ we have
$$w_j v=v e(u,z_j)\quad \text{for some }u\in\textstyle
\mathbb{Z},$$ 
where
$$e(u,z_j)=\exp( (q^u-q^{-u})z_j)=\sum_{r\geq
0}\frac{(q^{u} -
q^{-u})^r}{r!}z_{j}^{r}\in\mathbb{C}(q^{1/2})[[z_1,\ldots,z_n]].$$ Let
$\mathcal{C}$ be a set of all nonzero vectors
\begin{equation*}
c=f_{l_1}\cdots f_{l_r}v_{\lambda_i}e(u_1,z_{j_1})\cdots
e(u_s ,z_{j_s}), 
\end{equation*}
where 
\begin{align*}
&l_1,\ldots,l_r\in\left\{1,2,\ldots,n\right\},\,\,\, 1\leq j_1\leq
\ldots\leq j_s\leq n,\,\,\,  u_{1},\ldots,u_s\in\mathbb{Z},\,\,\, r,s\geq 0,\\
& u_{l}\geq u_{m}\quad \text{if}\quad l\leq m\text{ and } j_l =j_m.
\end{align*}
Recall \eqref{ordering} and set
$$o(c)=(-u_1,\ldots,-u_s),\qquad t(c)=(j_1,\ldots,j_s).$$ 
The set $\mathcal{C}_{i}:=\mathcal{C}\cap  L(\lambda_i)_z$ forms a basis of the
space $L(\lambda_i)_z$.

 There exists  a unique isomorphism of vector spaces $\Omega\colon
 L(\lambda_i)_z\to \left<\mathcal{Y}_i (z)\right>$ satisfying the following two
 conditions:
 \begin{enumerate}
   \item For all $f_{l_1}\cdots f_{l_r}v_{\lambda_i}e(u_1,z_{j_1})\cdots
e(u_s ,z_{j_s})\in \mathcal{C}_i$ there exist $t_1 ,\ldots, t_s
\in\frac{1}{2}\mathbb{Z}$ such that
\begin{align*}
&\Omega(f_{l_1}\cdots f_{l_r}v_{\lambda_i}e(u_1,z_{j_1})\cdots
e(u_s ,z_{j_s}))\\
&\qquad\qquad=x_{l_1}^{-}(z)_\bullet  \ldots_\bullet x_{l_r}^{-}(z)_\bullet
\mathcal{Y}_i (z)
\psi_{j_1}(zq^{t_1})\cdots\psi_{j_s}(zq^{t_s})\in\mathcal{B}_{i}\text{;}
\end{align*}
\item For all $c_1,c_2\in\mathcal{C}_{i}$ we have
$$\text{if }\wt(c_1)=\wt(c_2),\text{ }t(c_1)=t(c_2)\text{ and }o(c_1)\leq
o(c_2)\quad\text{then}\quad o(\Omega(c_1))\leq o(\Omega(c_2)),$$
where ``$\leq$'' is lexicographic order.
 \end{enumerate}
Notice that 
$$\wt ( \Omega(c)) =\wt(c) $$ 
for all $c\in\mathcal{C}_i$ and
 $$u_l -u_m =-(n+1)(t_l-t_m)\quad\text{when}\quad j_l=j_m.$$
 
The actions of the operators $\bar{e}_j$, $f_j$, $\bar{k}_j$ on an arbitrary
basis vector $c\in \mathcal{C}_i$ correspond to the actions of operators
$x_{j}^{+}(z)$, $x_{j}^{-}(z)$, $\psi_{j}(z)$ on $b(z)=\Omega(c)$ respectively.
For example, 
if $f_j c\neq 0$ for some $c\in\mathcal{C}_i$, then we have commutative diagrams
as in Figure \ref{kd1}. The left diagram is a consequence of \eqref{Q3}, while
the right diagram is a consequence of Lemma \ref{drinfeld}. \begin{figure}[H]
\begin{align*}\xymatrix @+4pc{
c e(u,z_j) \ar[r]^{f_j}  & f_j c e(u,z_j) \\
 c\ar[r]^{f_j}\ar[u]^{\frac{\bar{k}_j}{q-q^{-1}}} & f_j
 c\ar[u]_{-\frac{\bar{k}_j}{q-q^{-1}}}\ar[lu]_{\bar{e}_j}}
\qquad
\xymatrix @+4pc{
b(z)\psi_{j}(zq^t) \ar[r]^{x_{j}^{-}(z)}  & x_{j}^{-}(z)_\bullet
b(z)\psi_{j}(zq^t)
\\
 b(z)\ar[r]^{x_{j}^{-}(z)}\ar[u]^{\frac{\psi_{j}(z)}{q-q^{-1}}} &
 x_{j}^{-}(z)_\bullet
 b(z)\ar[u]_{-\frac{\psi_{j}(z)}{q-q^{-1}}}\ar[lu]_{x_{j}^{+}(z)}}
 \end{align*}
\caption{Commutative diagrams in  $L(\lambda_i)_z$ and $\left<\mathcal{Y}_i (z)\right>$}
\label{kd1} 
\end{figure}
\noindent Finally, we conclude that formulas \eqref{main_1}--\eqref{main_3}
define an $U_{q}(\mathfrak{sl}_{n+1})_z$-module structure on the space
$\left<\mathcal{Y}_i (z)\right>$, so the mapping $\Omega$ becomes an
$U_{q}(\mathfrak{sl}_{n+1})_z$-module isomorphism.
\end{proof}

\begin{rem}
For $i=0$ we have $\mathcal{Y}_{i}(z)=\mathcal{Y}_{0}(z)=1$. The action of
the Frenkel-Jing operators on $\mathcal{Y}_{0}(z)$ is trivial, i.e.
$$a(z)_\bullet \mathcal{Y}_{0}(z) =0\quad\text{for
}a(z)=x_{j}^{\pm}(z),\psi_{j}(z),\phi_{j}(z),\, j=1,2,\ldots,n,$$
so the space $\left<\mathcal{Y}_{0}(z)\right>$ is one-dimensional.
\end{rem}

In the end, we would like to provide an example of
$U_{q}(\mathfrak{sl}_{n+1})_z$-modules.
The $U_{q}(\mathfrak{sl}_{n+1})_z$-module
$L(\lambda_i)_z$, when considered as a vector space over
$\mathbb{C}(q^{1/2})$,  has a weight decomposition $$L(\lambda_i)_z =
\bigoplus_{\mu\in\wt L(\lambda_i)} \left(L(\lambda_i)_z\right)_{\mu},\quad\text{where}\quad\left(L(\lambda_i)_z\right)_{\mu}=\left\{v\in
L(\lambda_i)_z : \wt v=\mu\right\}.$$
Its weight subspaces $\left(L(\lambda_i)_z\right)_{\mu}$, $\mu\in\wt
L(\lambda_i)$, are infinite-dimensional.

\begin{ex}
Consider the diagrams of $U_{q}(\mathfrak{sl}_{2})_z$-module
$L(\lambda_1)_z$ and $U_{q}(\mathfrak{sl}_{3})_z$-module
$L(\lambda_1)_z$ given in Figures \ref{sl2zq} and \ref{sl3zq} respectively. Each
node represents one weight vector from the corresponding module. This vector  is also
an element of the basis $\mathcal{C}_1$, that was constructed in the proof of
Theorem \ref{main}.
Weight of each node is written as its label and all the nodes (i.e. corresponding vectors) are linearly
independent.
The yellow node represents the highest weight vector $v_{\lambda_1}\in
L(\lambda_1)\subset L(\lambda_1)_z$.
The arrows represent (nonzero) actions of the
generators $\bar{e}_{j},f_j,\bar{h}_j =\bar{k}_j/(q-q^{-1})$, where $j=1$
(Figure \ref{sl2zq}) or $j=1,2$ (Figure \ref{sl3zq}), on the corresponding nodes (vectors).
All the diagrams in the both figures are commutative.

The images of the bases $\mathcal{C}_1$, given in Figures 
\ref{sl2zq} and \ref{sl3zq}, under the $U_{q}(\mathfrak{sl}_{n+1})_z$-module
isomorphism $\Omega$, constructed in the proof of Theorem
\ref{main}, where $n=1,2$, are equal to the bases $\mathcal{B}_{1}$, given in
 Examples \ref{exampleeasy} and \ref{examplehard}.
\end{ex}

\begin{figure}[H]
\centering
\begin{tikzpicture}[scale=1.7]
\tikzstyle{every node}=[font=\tiny,inner sep=0pt,minimum size=31pt]
\tikzstyle{sh}=[-,line width=8pt,white]
\node at ( 0,0) [shape=circle,draw,fill=yellow] (a)
{$\lambda_1$}; \node at ( 2,0) [shape=circle,draw] (b) {$\lambda_1 \hspace{-2pt}-\hspace{-2pt}\alpha_1$};
\node at ( 0,2) [shape=circle,draw] (c) {$\lambda_1$};
\node at ( 2,2)  (d) {};
\draw [->,blue,thick] (a) -- (b);
\draw [->,blue,thick] (a) -- (c);
\node at ( 2,2) [shape=circle,draw] (d) {$\lambda_1 \hspace{-2pt}-\hspace{-2pt}\alpha_1$};
\draw [sh] (b) -- (d);
\draw [->,blue,thick] (b) -- (d);
\draw [sh] (c) -- (d);
\draw [->,blue,thick] (c) -- (d);
\draw [->,blue,thick] (b) -- (c);
\node at ( 0,4) (x) [shape=circle,draw] {$\lambda_1$};
\node at ( 2,4)  (y) [shape=circle,draw] {$\lambda_1 \hspace{-2pt}-\hspace{-2pt}\alpha_1$};
\draw [->,blue,thick] (c) -- (x);
\draw [sh] (d) -- (x);
\draw [->,blue,thick] (d) -- (x);
\draw [sh] (d) -- (y);
\draw [sh] (d) -- (y);
\draw [->,blue,thick] (d) -- (y);
\tikzstyle{qw}=[circle,draw=white!10,fill=white!10, draw opacity=0.6, fill
opacity=0.6,minimum size=0.6cm]
\node at (1,0) [qw] {} ; 
\node at (1,0) {$f_1$} ; 
\node at (1,1) [qw] {} ;  
\node at (1,1) {$\bar{e}_1$} ;
\node at (1,2) [qw] {} ;
\node at (1,2)  {$f_1$} ;
\node at (0,1) [qw] {} ;
\node at (0,1)  {$\bar{h}_1$} ;
\node at (2,1) [qw] {} ;
\node at (2,1)  {$-\bar{h}_1$} ;
\node at (0,3) [qw] {} ;
\node at (0,3)  {$\bar{h}_1$} ;
\node at (2,3) [qw] {} ;
\node at (2,3)  {$-\bar{h}_1$} ;
\node at (1,3) [qw] {} ;
\node at (1,3)  {$\bar{e}_1$} ;
\draw [->,blue,thick] (x) -- (y);
\node at ( 0,6)  (xx) {$\vdots$};
\node at ( 2,6)  (yy) {$\vdots$};
\draw [->,blue,thick] (x) -- (xx);
\draw [->,blue,thick] (y) -- (yy);
\draw [->,blue,thick] (y) -- (xx);
\node at (1,5) [qw] {} ;
\node at (1,5)  {$\bar{e}_1$} ;
\node at (2,5) [qw] {} ;
\node at (2,5)  {$-\bar{h}_1$} ;
\node at (0,5) [qw] {} ;
\node at (0,5)  {$\bar{h}_1$} ;
\node at (1,4) [qw] {} ;
\node at (1,4)  {$f_1$} ;
\end{tikzpicture}\caption{$U_{q}(\mathfrak{sl}_{2})_z$-module
$L(\lambda_1)_z$}
\label{sl2zq}
\end{figure}
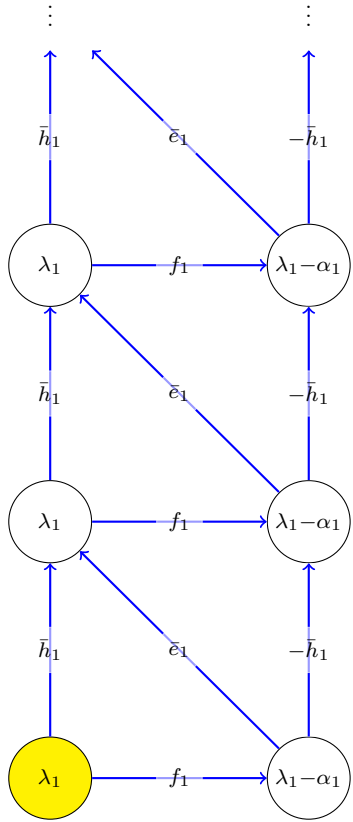

\begin{figure}[H]
\centering
\begin{tikzpicture}[scale=2.6]
\tikzstyle{every node}=[font=\tiny,inner sep=0pt,minimum size=31pt]
\tikzstyle{sh}=[-,line width=8pt,white]
\node at ( 0,0) [shape=circle,draw,fill=yellow] (a)
{$\lambda_1$}; \node at ( 2,0) [shape=circle,draw] (b) {$\lambda_1 \hspace{-2pt}-\hspace{-2pt}\alpha_1$};
\node at ( 0,2) [shape=circle,draw] (c) {$\lambda_1$};
\node at ( 2,2)  (d) {};
\node at ( 2.8,0.8) [shape=circle,draw] (bb) {$\lambda_1 \hspace{-2pt}-\hspace{-2pt}\alpha_1$};
\node at ( 0.8,2.8) (cc) {};
\node at ( 2.8,2.8) [shape=circle,draw] (dd) {$\lambda_1 \hspace{-2pt}-\hspace{-2pt}\alpha_1$};
\draw [->,blue,thick] (a) -- (b);
\draw [->,blue,thick] (a) -- (c);
\node at ( 2,2) [shape=circle,draw] (d) {$\lambda_1 \hspace{-2pt}-\hspace{-2pt}\alpha_1$};
\node at ( 0.8,2.8) [shape=circle,draw] (cc) {$\lambda_1$};
\draw [->,blue,thick] (bb) -- (cc);
\draw [sh] (b) -- (d);
\draw [->,blue,thick] (b) -- (d);
\draw [sh] (c) -- (d);
\draw [->,blue,thick] (c) -- (d);
\draw [->,blue,thick] (b) -- (c);
\draw [->,blue,thick] (bb) -- (dd);
\draw [->,blue,thick] (cc) -- (dd);
\node at ( 0.8,4.8) (z) {$\vdots$};
\node at ( 2.8,4.8)  (w) {$\vdots$};
\node at ( 0,4) (x) {$\vdots$};
\node at ( 2,4)  (y) {$\vdots$};
\draw [->,blue,thick] (c) -- (x);
\draw [->,blue,thick] (cc) -- (z);
\draw [sh] (d) -- (x);
\draw [->,blue,thick] (d) -- (x);
\draw [->,blue,thick] (dd) -- (z);
\draw [sh] (d) -- (y);
\draw [sh] (d) -- (y);
\draw [->,blue,thick] (d) -- (y);
\draw [->,blue,thick] (dd) -- (w);

\node at ( 4,0) [shape=circle,draw] (bbb)
{$\begin{matrix}\lambda_1
\hspace{-2pt}-\hspace{-2pt}\alpha_1\\-\alpha_2\end{matrix}$}; \node
at ( 4,2) [shape=circle,draw] (ddd) {$\begin{matrix}\lambda_1 \hspace{-2pt}-\hspace{-2pt}\alpha_1\\-\alpha_2\end{matrix}$}; \node at ( 4.8,0.8) [shape=circle,draw] (bbbb) {{$\begin{matrix}\lambda_1
\hspace{-2pt}-\hspace{-2pt}\alpha_1\\-\alpha_2\end{matrix}$}}; \node at ( 4.8,2.8) [shape=circle,draw]
(dddd) {{$\begin{matrix}\lambda_1
\hspace{-2pt}-\hspace{-2pt}\alpha_1\\-\alpha_2\end{matrix}$}};
\draw [->,red,thick] (b) -- (bb); \draw [->,red,thick] (d) -- (dd); \draw [->,red,thick] (bbb) -- (bbbb); \draw [->,red,thick] (ddd) -- (dddd);
\draw [->,red,thick] (b) -- (bbb);
\draw [sh] (d) -- (ddd);

\draw [->,red,thick] (d) -- (ddd);
\draw [->,red,thick] (bb) -- (bbbb);
\draw [->,red,thick] (dd) -- (dddd);
\draw [->,red,thick] (bbb) -- (bb);
\draw [->,red,thick] (ddd) -- (dd);
\node at ( 3.6,1.6)  (xx) {$\iddots$};
\node at ( 3.6,3.6)  (zz) {$\iddots$};
\node at ( 5.6,1.6)  (yy) {$\iddots$};
\node at ( 5.6,3.6)  (ww) {$\iddots$};
\draw [->,red,thick] (bb) -- (xx);
\draw [->,red,thick] (bbbb) -- (yy);
\draw [->,red,thick] (dd) -- (zz);
\draw [->,red,thick] (dddd) -- (ww);
\draw [->,red,thick] (dddd) -- (zz);
\draw [->,red,thick] (bbbb) -- (xx);
\tikzstyle{qw}=[circle,draw=white!10,fill=white!10, draw opacity=0.6, fill
opacity=0.6,minimum size=0.6cm]
\node at (1,0) [qw] {} ; 
\node at (1,0) {$f_1$} ; 
\node at (1,1) [qw] {} ; 
\node at (1,1) {$\bar{e}_1$} ;
\node at (1,2) [qw] {} ;
\node at (1,2)  {$f_1$} ;
\node at (0,1) [qw] {} ;
\node at (0,1)  {$\bar{h}_1$} ;
\node at (2,1) [qw] {} ;
\node at (2,1)  {$-\bar{h}_1$} ;
\node at (0,3) [qw] {} ;
\node at (0,3)  {$\bar{h}_1$} ;
\node at (2,3) [qw] {} ;
\node at (2,3)  {$-\bar{h}_1$} ;
\node at (1,3) [qw] {} ;
\node at (1,3)  {$\bar{e}_1$} ;
\node at (1.8,1.8) [qw] {} ;
\node at (1.8,1.8)  {$\bar{e}_1$} ;
\node at (1.8,2.8) [qw] {} ;
\node at (1.8,2.8)  {$f_1$} ;
\node at (2.8,1.8) [qw] {} ;
\node at (2.8,1.8)  {$-\bar{h}_1$} ;
\node at (0.8,3.8) [qw] {} ;
\node at (0.8,3.8)  {$\bar{h}_1$} ;
\node at (2.8,3.8) [qw] {} ;
\node at (2.8,3.8)  {$-\bar{h}_1$} ;
\node at (1.8,3.8) [qw] {} ;
\node at (1.8,3.8)  {$\bar{e}_1$} ;
 
\node at (3,0) [qw] {} ; 
\node at (3,0) {$f_2$} ; 
\node at (2.4,0.4) [qw] {} ; 
\node at (2.4,0.4) {$\bar{h}_2$} ; 
\node at (4.4,0.4) [qw] {} ; 
\node at (4.4,0.4) {$-\bar{h}_2$} ;
\node at (3.8,0.8) [qw] {} ; 
\node at (3.8,0.8) {$f_2$} ;  
\node at (3.4,0.4) [qw] {} ; 
\node at (3.4,0.4) {$\bar{e}_2$} ;
\node at (3.2,1.2) [qw] {} ; 
\node at (3.2,1.2) {$\bar{h}_2$} ;
\node at (5.2,1.2) [qw] {} ; 
\node at (5.2,1.2) {$-\bar{h}_2$} ;
\node at (4.2,1.2) [qw] {} ; 
\node at (4.2,1.2) {$\bar{e}_2$} ;
\node at (3,2) [qw] {} ; 
\node at (3,2) {$f_2$} ; 
\node at (2.4,2.4) [qw] {} ; 
\node at (2.4,2.4) {$\bar{h}_2$} ; 
\node at (4.4,2.4) [qw] {} ; 
\node at (4.4,2.4) {$-\bar{h}_2$} ;
\node at (3.8,2.8) [qw] {} ; 
\node at (3.8,2.8) {$f_2$} ;  
\node at (3.4,2.4) [qw] {} ; 
\node at (3.4,2.4) {$\bar{e}_2$} ;
\node at (3.2,3.2) [qw] {} ; 
\node at (3.2,3.2) {$\bar{h}_2$} ;
\node at (5.2,3.2) [qw] {} ; 
\node at (5.2,3.2) {$-\bar{h}_2$} ;
\node at (4.2,3.2) [qw] {} ; 
\node at (4.2,3.2) {$\bar{e}_2$} ;
\node at ( 2,2) [shape=circle,draw,thin] (d) {}; 
\node at ( 0.8,2.8) [shape=circle,draw, thin] (cc) {};

\end{tikzpicture}\caption{$U_{q}(\mathfrak{sl}_{3})_z$-module
$L(\lambda_1)_z$}
\label{sl3zq}
\end{figure}
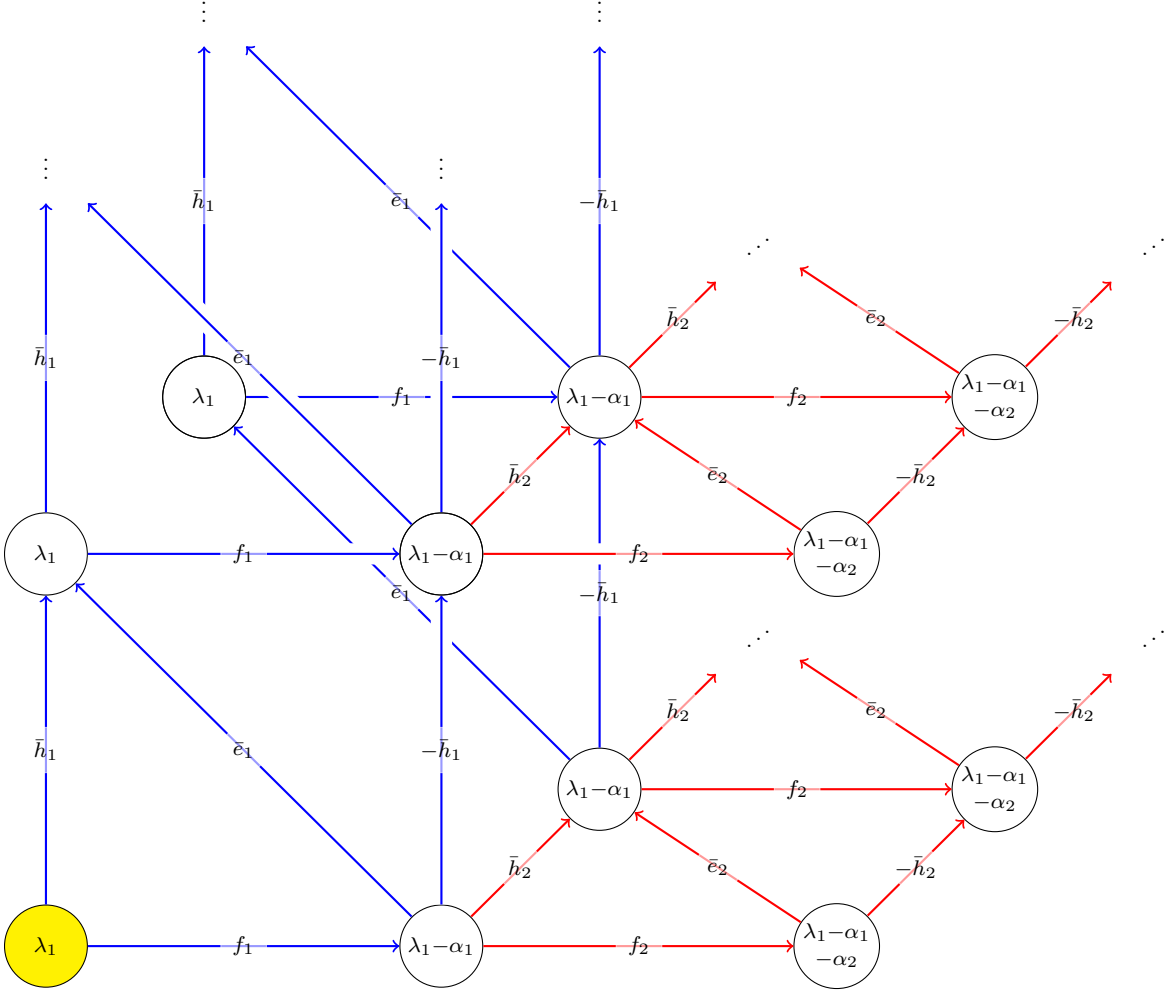


\section*{Acknowledgement}
The author would like to thank Mirko Primc for his valuable comments and
suggestions.


\end{document}